\documentclass[preprint]{elsarticle}
\usepackage{amsmath,amssymb,amsfonts,amsthm,amsopn,dsfont}
\usepackage{bm}
\usepackage{graphicx}
\usepackage{float}
\newtheorem{prop}{Proposition}
\newtheorem{lemma}{Lemma}
\newtheorem{rem}{Remark}

\usepackage[utf8]{inputenc}
\usepackage[T1]{fontenc}

\usepackage{xcolor}

\usepackage[ruled,vlined,linesnumbered]{algorithm2e}
\usepackage{multirow}

\newcommand{\hsp}{\mathcal{H}^{\Gamma}}

\newcommand{\extg}{\mathcal{E}_{_{\Gamma}}}
\newcommand{\exts}{\mathcal{E}_{_{\Sigma}}}
\newcommand{\trg}{\gamma_{_{\Gamma}}}
\newcommand{\extgi}{\mathcal{E}_{_{\Gamma_i}}}
\newcommand{\extsi}{\mathcal{E}_{_{\Sigma_i}}}
\newcommand{\trgi}{\gamma_{_{\Gamma_i}}}

\begin{document}

\begin{frontmatter}

\title{3D-1D coupling on non conforming meshes via three-field optimization based domain decomposition}

\author[polito]{Stefano Berrone \corref{cor}}
\ead{stefano.berrone@polito.it}
\author[polito]{Denise Grappein}
\ead{denise.grappein@polito.it}
\author[polito]{Stefano Scial\`o}
\ead{stefano.scialo@polito.it}

\address[polito]{Dipartimento di Scienze Matematiche, Politecnico di Torino,
	Corso Duca degli Abruzzi 24, 10129 Torino, Italy. Member of INdAM research group GNCS.}
\cortext[cor]{Corresponding author}

\begin{abstract}
A new numerical approach is proposed for the simulation of coupled three-dimensional and one-dimensional elliptic equations (3D-1D coupling) arising from dimensionality reduction of 3D-3D problems with thin inclusions. The method is based on a well posed mathematical formulation and results in a numerical scheme with  high robustness and flexibility in handling geometrical complexities. This is achieved by means of a three-field approach to split the 1D problems from the bulk 3D problem, and then resorting to the minimization of a properly designed functional to impose matching conditions at the interfaces. Thanks to the structure of the functional, the method allows the use of independent meshes for the various subdomains.
\end{abstract}


\end{frontmatter}

\section{Introduction} This work presents a new numerical approach to manage the coupling of three-dimensional and one-dimensional elliptic equations (3D-1D coupling). This kind of problems emerges, for example, in the numerical treatment of domains with small tubular inclusions: in these cases, indeed, it might result computationally convenient to approximate the small inclusions by one-dimensional (1D) manifolds, in order to avoid the building of a three-dimensional grid within the inclusion. Clearly this topological reduction can be a viable approach only if one-dimensional modeling assumptions can be applied to the problem at hand. Examples of application are: capillary networks exchanging flux with the surrounding tissue \cite{Notaro2016,Koppl2020}, the interaction of tree roots with the soil \cite{Schroder2012,Koch2018}, a system of wells for fluid in geological applications \cite{Gjerde2018a,Gjerde2020, CerLauZun2019}, or the modeling of fiber-reinforced materials \cite{Steinbrecher2020,Llau2016}.

The definition of coupling conditions between a three-dimensional (3D) and a 1D problem is not straightforward, as no bounded trace operator is defined in standard functional spaces on manifolds with a dimensionality gap higher than one. This problem was recently studied in \cite{Dangelo2012,DangeloQuarteroni2008}, where suitable weighed Sobolev spaces were introduced and a bounded trace operator from the 3D space to the 1D space was defined, thus allowing to formulate a well posed coupled problem, also resorting to the results in \cite{ErnGuermond}. In \cite{Tornberg2004} regularizing techniques are proposed for singular terms. Three dimensional problems with singular sources defined on lines are also studied in \cite{Gjerde2019}, where the nature of the irregularity is analyzed and a method based on the splitting of the solution in a low regularity part plus a regular correction is proposed. Problems with a source term on manifolds with high dimensionality gaps are also studied in \cite{Koppl2018}, where a lifting technique of the irregular datum is used to reduce the dimensionality gap. In \cite{Zunino2019} a 3D-1D coupled approach is derived starting from the fully 3D-3D coupled problem and applying a topological model reduction through the definition of proper averaging operators. 

In the present work a new numerical approach to this problem is proposed, starting from modeling assumptions similar to those proposed in \cite{Zunino2019}. A reduced 3D-1D model approximating the original equi-dimensional 3D-3D problem is here obtained by introducing proper assumptions on the solution inside the small inclusions and by defining suitable subspaces of the Sobolev spaces typically employed in the variational formulation of partial differential equations. Thanks to this, the problem is treated, in practice, as a 3D-1D reduced problem, but it can still be written as an equi-dimensional problem, thus skipping the difficulties related to the 3D-1D coupling. The problems in the bulk 3D domain and in the small inclusions are splitted resorting to a three-field based domain decomposition method, originally formulated in \cite{Brezzi} and already applied for domain decomposition in networks of fractures \citep{NostroArxiv}. Suitable matching conditions are then enforced at the interface to recover the solution on the whole domain. Here pressure continuity and flux conservation constraints are assumed at the interfaces. The advantages of the three-field based approach with respect to standard domain decomposition methods lie in the possibility of defining stable locally conservative numerical schemes on non conforming meshes \citep{NostroArxiv}. Matching conditions at the interfaces are enforced by means of a PDE constrained numerical scheme, already proposed for simulation of the flow in poro-fractured media \cite{BPSc,BSV,BDS} and here proposed for 3D-1D coupling. The method is based on the minimization of a cost functional expressing the error in the fulfillment of interface conditions, constrained by constitutive laws on the various subdomains. The advantages of such an approach lie in the possibility of enforcing continuity and flux balance at the interfaces using completely independent meshes on the sub-domains. This provides to the scheme an extreme flexibility and robustness to geometrical complexities, which are critical features for the applicability to real problems, where the nearly one dimensional inclusions might form complex networks. Exploiting the properties of the functional spaces chosen for the solution, the functional can be reduced to the centrelines of the tubular inclusions and used to control the continuity of the solution, whereas flux conservation is strongly enforced thanks to the three-field formulation.

The manuscript is organized as follows: in Section~\ref{not_and_form} notation is introduced and the strong formulation of the 3D-3D problem is presented, along with the hypotheses allowing its reduction to a 3D-1D problem. In Section~\ref{VarForm} the weak formulation of the 3D-1D coupled problem is worked out, while in Section~\ref{PDE_constr} the problem is re-written into a PDE-constrained optimization formulation. The corresponding discrete approach is discussed in Section~\ref{Discrete}. Finally, in Section~\ref{Num_res}, some numerical examples are described.

\section{Notation and problem formulation}\label{not_and_form}
Let us here briefly recall the basic formulation of the problem of interest in a simplified setting, the ideas here proposed being easily extendable to more general cases. We refer to \citep{Zunino2019} for a broader presentation of the problem.
Let us consider a convex domain $\Omega \in \mathbb{R}^3$ in which a generalized cylinder $\Sigma \in \mathbb{R}^3$ is embedded. The centreline of this cylinder is denoted by $\Lambda=\left\lbrace\bm{\lambda}(s), s \in (0,S) \right\rbrace$, where $\bm{\lambda}(s)$ is here assumed, for simplicity, to be a rectilinear segment in the three-dimensional space. The symbol $\Sigma(s)$ denotes the transversal section of the cylinder at $s \in [0,S]$. We assume that each section, whose boundary is denoted by $\Gamma(s)$, has an elliptic shape, denoting by $R$ the maximum axes length of the ellipses as $s$ ranges in the interval $[0,S]$. The lateral surface of $\Sigma$ is $\Gamma=\left\lbrace \Gamma(s), ~s \in [0,S]\right\rbrace $, while $\Sigma_0=\Sigma(0)$ and $\Sigma_S=\Sigma(S)$ are the two extreme sections.  The portion of the domain that does not include the cylinder is denoted by $D=\Omega \setminus \Sigma$, with boundary $\partial D=\partial \Omega \cup \left\lbrace \Gamma\cup \Sigma_0 \cup \Sigma_S\right\rbrace $, where $\partial \Omega$ is the boundary of $\Omega$. We refer to $\partial D^e$ as the external boundary of $D$, coinciding with $\partial \Omega$ when the extreme sections of $\Sigma$ are inside $\Omega$. In case $\Sigma_0$ and $\Sigma_S$ lie on the boundary $\partial \Omega$ we define $\partial D^e=\partial \Omega \setminus \left\lbrace \Sigma_0 \cup \Sigma_S\right\rbrace$.

Let us consider, in $\Omega$, a diffusion problem, with unknown pressures $u$ in $D$ and $\tilde{u}$ in $\Sigma$:\medskip

\begin{minipage}{0.47\textwidth}
	\begin{center}
		\textbf{3D-problem on} $\bm{D}$:
		\begin{align}
		-&\nabla \cdot (\bm{K} \nabla u)=f & \text{in } D\label{eqOmega}\\[0.35em]
		&u=0 &\text{on } \partial D^e \label{cbOmega}\\
			&u_{|_{\Gamma}}=\psi &~\text{ on } \Gamma \label{psi_u}\\
		&\bm{K}\nabla u \cdot \bm{n}=\phi &~\text{on } \Gamma \label{phi_u}
		\end{align}		
	\end{center}
\end{minipage}
\begin{minipage}{0.47\textwidth}
	\begin{center}
		\textbf{3D-problem on} $\bm{\Sigma}$:
		\begin{align}
		-&\nabla \cdot(\tilde{\bm{K}}\nabla\tilde{u})=g&~\text{ in } \Sigma \label{eqSigma}\\
		&\tilde{u}=0 &\text{ on } \Sigma_0 \cup \Sigma_S \label{cbSigma}\\
			&\tilde{u}_{|_{\Gamma}}=\psi &~\text{ on } \Gamma \label{psi_utilde}\\
		&\tilde{\bm{K}}\nabla \tilde{u}\cdot \bm{\tilde{n}}=-\phi &~\text{on } \Gamma\label{phi_utilde}	
		\end{align}	
	\end{center}
\end{minipage}
\medskip

\noindent Vectors $\bm{n}$ and $\bm{\tilde{n}}$ are the outward pointing unit normal vectors to $\Gamma$, respectively for $D$ and $\Sigma$, such that $\bm{\tilde{n}}=-\bm{n}$, $\bm{K}$ and $\tilde{\bm{K}}$ are uniformly positive definite tensors in $D$ and $\Sigma$, respectively, and $f$ and $g$ are source terms. The symbol $\psi$ denotes the unknown unique value of the pressure on the interface $\Gamma$ while $\phi$ is the unknown flux through $\Gamma$, entering in $D$. Homogeneous Dirichlet boundary conditions are considered on $\partial D ^e$ and on $\Sigma_0$ and $\Sigma_S$, assumed to be part of $\partial \Omega$, the extension to more general cases being straightforward.
Equations (\ref{psi_u})-(\ref{phi_u}) and (\ref{psi_utilde})-(\ref{phi_utilde}) enforce pressure continuity and flux conservation constraints on the lateral surface $\Gamma$ of the cylinder, thus allowing us to couple the two problems. We remark that different coupling conditions could be considered, as the ones proposed in \cite{Zunino2019}.

If the cross-section-size $R$ of the cylinder $\Sigma$ becomes much smaller than the characteristic dimension of $\Omega$, a model can be introduced in order to reduce the computational cost of simulations. The key point is that, as $R\ll \text{diam}(\Omega)$, it is possible to assume that the variations of $\tilde{u}$ on the transversal sections of the cylinder are negligible, i.e., in cylindrical coordinates,
\begin{equation}
\tilde{u}(r,\theta,s)=\hat{u}(s) \quad\forall r\in [0,R],~ \forall \theta \in [0,2\pi). \label{hpcost}
\end{equation}
This allows us to simplify problem \eqref{eqSigma}-\eqref{cbSigma} reducing it to a 1D problem defined on the cylinder's centerline $\Lambda$ as
\begin{align}
&-\cfrac{d}{ds}\left( \tilde{\bm{K}}|\Sigma(s)|\cfrac{d\hat{u}}{ds}\right) =\tilde{g} & \text { for } s \in (0,S)\label{eqlambda}\\
&\hat{u}(0)=\hat{u}(S)=0,\label{cblambda}
\end{align}
where the new forcing term $\tilde{g}$ now accounts for the original volumetric source $g$ and for the incoming flux from the boundary $\Gamma$ of the equi-dimensional problem.
As the domain of problem \eqref{eqSigma}-\eqref{cbSigma} is reduced to a 1D segment, the domain of problem \eqref{eqOmega}-\eqref{cbOmega} is extended to $\Omega\setminus\Lambda$, thus obtaining a coupled 3D-1D problem. Details of this geometrical reduction are provided in the next session.
The clear advantage of the reduced problem is that solving a problem defined on a segment instead of a problem defined on a small cylinder is computationally much cheaper. Nevertheless it is not straightforward to write coupling conditions between \eqref{eqlambda}-\eqref{cblambda} and \eqref{eqOmega}-\eqref{cbOmega} that are analogous to \eqref{psi_u}-\eqref{phi_u} and \eqref{psi_utilde}-\eqref{phi_utilde}, as there is no bounded trace operator from $H^1(\Omega)$ to $L^2(\Lambda)$, being the dimensionality gap between the interested manifolds higher than one. 

In the next section an original formulation of the coupled 3D-1D problem is derived, starting from a variational formulation of the fully dimensional 3D-3D problem \eqref{eqOmega}-\eqref{cbOmega} and \eqref{eqSigma}-\eqref{cbSigma}, and defining the proper functional spaces and operators required to reduce this formulation to a well posed 3D-1D coupling.

\section{Variational formulation}\label{VarForm}
In this section we will adopt the following notation:
\begin{align*}
&H_0^1(D)=\left\lbrace v \in H^1(D) : v_{|_{\partial D^e}}=0 \right\rbrace\\ 
&H_0^1(\Sigma)=\left\lbrace v \in H^1(\Sigma): v_{|_{\Sigma_0}}=v_{|_{\Sigma_S}}=0\right\rbrace \\
&H_0^1(\Lambda)=\left\lbrace v \in H^1(\Lambda): v(0)=v(S)=0\right\rbrace 
\end{align*}
and we will suppose that $\partial D^e =\partial \Omega \setminus \left\lbrace \Sigma_0\cup \Sigma_S\right\rbrace$, i.e. the extreme sections of the cylinder lie on $\partial \Omega$.
Let us define a trace operator
\begin{equation}
\gamma_{_\Gamma}:H^1(D)\cup H^1(\Sigma)\rightarrow H^{\frac{1}{2}}(\Gamma),\text{ such that }\gamma_{_\Gamma}v=v_{|_\Gamma} ~\forall v \in H^1(D)\cup H^1(\Sigma)
\end{equation} and two extension operators 
\begin{equation*}
\mathcal{E}_{_\Sigma}: H^1(\Lambda) \rightarrow H^1(\Sigma) ~\text{  and  }~\mathcal{E}_{_\Gamma}: H^1(\Lambda) \rightarrow H^{\frac{1}{2}}(\Gamma)
\end{equation*}
defined such that, given $\hat{v} \in H_0^1(\Lambda)$, $\mathcal{E}_{_\Sigma}(\hat{v})$ is the extension of the point-wise value $\hat{v}(s)$, $s \in [0,S]$, to the cross section $\Sigma(s)$ of the cylinder and $\mathcal{E}_{_\Gamma}(\hat{v})$ the extension to the boundary $\Gamma(s)$ of the cross section. Let us observe that $\mathcal{E}_{_\Gamma}=\gamma_{_\Gamma}\circ \mathcal{E}_{_\Sigma}$.
Setting $\hat{V}=H_0^1(\Lambda)$, let us further consider the spaces:
\begin{align*}
&\widetilde{V}=\lbrace v \in H_0^1(\Sigma): v =\exts\hat{v}, ~\hat{v} \in \hat{V} \rbrace, \\
&\mathcal{H}^{\Gamma}=\lbrace v \in H^{\frac{1}{2}}(\Gamma): v =\extg\hat{v}, ~\hat{v} \in \hat{V} \rbrace,\\
&V_D=\left\lbrace v \in H_0^1(D): \gamma_{_\Gamma}v \in \mathcal{H}^{\Gamma}\right\rbrace.
\end{align*} 
such that $\widetilde{V}\subset H_0^1(\Sigma)$ contains functions that are extensions to the whole $\Sigma$ of functions in  $\hat{V}$, $\mathcal{H}^{\Gamma}\subset H^{\frac{1}{2}}(\Gamma)$ contains functions that are extensions to $\Gamma$ of functions in $\hat{V}$, and $V_D\subset H_0^1(D)$ only contains functions whose trace on $\Gamma$ is a function of $\mathcal{H}^{\Gamma}$.
The variational problem arising from the coupling of \eqref{eqOmega}-\eqref{cbOmega} and \eqref{eqSigma}-\eqref{cbSigma} through the continuity constraint can be written as: 
\textit{find} $(u, \tilde{u}) \in V_D \times \widetilde{V}$ \textit{such that}
\begin{align}
&(\bm{K}\nabla u, \nabla v)_{V_D}+(\bm{\tilde{K}}\nabla \tilde{u}, \nabla \tilde{v})_{\widetilde{V}}=(f,v)_{V_D}+(g,\tilde{v})_{\widetilde{V}} ~~&\forall (v,\tilde{v} )\in V_D\times \widetilde{V}\label{eq}\\
&\left\langle \gamma_{_\Gamma}u-\gamma_{_\Gamma}\tilde{u},\eta\right\rangle_{\mathcal{H}^{\Gamma},{\mathcal{H}^{\Gamma}}'} =0, \qquad&\forall\eta \in {\mathcal{H}^{\Gamma}}'\label{continuita}
\end{align}
\begin{rem}
Let us consider the space $\mathbb{V}=\left\lbrace (v,\tilde{v}) \in V_D \times\widetilde{V}:\gamma_{_\Gamma}v=\gamma_{_\Gamma}\tilde{v}\right\rbrace$. Then, problem (\ref{eq})-(\ref{continuita}) is equivalent to:
\textit{find} $(u, \tilde{u}) \in \mathbb{V}$ \textit{such that}
\begin{equation*}
	(\bm{K}\nabla u, \nabla v)_{V_D}+(\bm{\tilde{K}}\nabla \tilde{u}, \nabla \tilde{v})_{\widetilde{V}}=(f,v)_{V_D}+(g,\tilde{v})_{\widetilde{V}} ~~\forall (v,\tilde{v} )\in \mathbb{V}
\end{equation*}
The well-posedness of the problem easily follows from Lax-Milgram theorem, considering $||\cdot||_{\mathbb{V}}=||\cdot||_{H^1(D)}+||\cdot||_{H^1(\Sigma)}$
\end{rem}
Equation \eqref{eq} can be split into two coupled equations introducing the unknown flux $\phi$ through $\Gamma$, obtaining thus
\begin{align}
&(\bm{K}\nabla u, \nabla v)_{V_D}-\left\langle \phi,\gamma_{_\Gamma}v \right\rangle_{{\mathcal{H}^{\Gamma}}', {\mathcal{H}^{\Gamma}}}=(f,v)_{V_D}~&\forall v \in V_D,~\phi \in {\mathcal{H}^{\Gamma}}' \label{eq_var_u} \\
&(\bm{\tilde{K}}\nabla \tilde{u}, \nabla \tilde{v})_{\widetilde{V}}+\left\langle \phi,\gamma_{_\Gamma}\tilde{v} \right\rangle_{{\mathcal{H}^{\Gamma}}', {\mathcal{H}^{\Gamma}}}=(g,\tilde{v})_{\widetilde{V}} &\forall \tilde{v} \in \widetilde{V},~\phi \in {\mathcal{H}^{\Gamma}}' \label{eq_var_uhat}
\end{align}
Moreover, the continuity condition (\ref{continuita}) can be rewritten introducing a new variable $\psi \in \mathcal{H}^{\Gamma}$ as:
\begin{align}
&\left\langle \gamma_{_\Gamma}u-\psi,\eta\right\rangle_{\mathcal{H}^{\Gamma},{\mathcal{H}^{\Gamma}}'}=0&~\forall \eta \in  {\mathcal{H}^{\Gamma}}', \psi \in \mathcal{H}^{\Gamma} \label{condpsi_u}\\
&\left\langle \gamma_{_\Gamma}\tilde{u}-\psi,\eta\right\rangle_{\mathcal{H}^{\Gamma},{\mathcal{H}^{\Gamma}}'}=0&~\forall \eta \in  {\mathcal{H}^{\Gamma}}', \psi \in \mathcal{H}^{\Gamma}\label{condpsi_hat}.
\end{align}
The set of equations (\ref{eq_var_u}), (\ref{eq_var_uhat}), (\ref{condpsi_u}) and (\ref{condpsi_hat}) represent an application of the 
three-field formulation presented in \cite{Brezzi} and similarly applied in \cite{NostroArxiv}.\smallskip\\
Thanks to the assumptions on the introduced functional spaces, this problem can be reduced to a 3D-1D coupled problem without encountering the aforementioned issues in the definition of a trace operator. In fact we only need to use the trace operator $\gamma_{_\Gamma}(\cdot)$, which is well-posed as defined from a three-dimensional manifold to a two dimensional one. Let us observe that
\begin{equation*}
 \left\langle \phi,\gamma_{_\Gamma}v \right\rangle_{{\mathcal{H}^{\Gamma}}', {\mathcal{H}^{\Gamma}}}=\int_{\Gamma}\phi~\gamma_{_\Gamma}v~d\Gamma=\int_0^S\Big( \int_{\Gamma(s)}\phi~\gamma_{_\Gamma}v~dl\Big) ds \quad \forall v \in V_D
\end{equation*}
and let us denote by $\overline{\phi}(s)$ the mean value of $\phi$ on the border $\Gamma(s)$ of each section. As $v \in V_D$ we know that $\gamma_{_\Gamma}v \in \mathcal{H}^{\Gamma}$, i.e. $\exists \check{v} \in \hat{V}: \gamma_{_\Gamma}v=\extg\check{v}$. Thus
$\int_{\Gamma(s)}\trg v  ~dl=|\Gamma(s)|\check{v}(s)$
and
\begin{equation*}
\int_0^S\Big( \int_{\Gamma(s)}\phi~\gamma_{_\Gamma}v~dl\Big) ds=\int_0^S|\Gamma(s)|\overline{\phi}(s)\check{v}(s)~ds=\left\langle |\Gamma|\overline{\phi},\check{v}\right\rangle_{\hat{V}',\hat{V}},
\end{equation*}
where $|\Gamma(s)|$ is the section perimeter size at $s\in[0,S]$.
The same holds if we consider $ \left\langle \phi,\gamma_{_\Gamma}\tilde{v} \right\rangle_{{\mathcal{H}^{\Gamma}}', {\mathcal{H}^{\Gamma}}}$. As $\tilde{v} \in \widetilde{V}$ we know that $\exists \hat{v} \in \hat{V}: \tilde{v}=\exts\hat{v}$ and consequently $\gamma_{_\Gamma}\tilde{v}=\gamma_{_\Gamma}\exts\tilde{v}=\extg\hat{v}$, so that
\begin{equation*}
\left\langle \phi,\gamma_{_\Gamma}\tilde{v} \right\rangle_{{\mathcal{H}^{\Gamma}}', {\mathcal{H}^{\Gamma}}}=\left\langle |\Gamma|\overline{\phi},\hat{v}(s)\right\rangle_{\hat{V}',\hat{V}}.
\end{equation*}
Similarly, $\forall \eta \in{\mathcal{H}^{\Gamma}}'$ and $\forall \rho: \rho=\extg\hat{\rho}$ with $\hat{\rho} \in \hat{V}$, we can write
\begin{equation}
\left\langle \rho,\eta\right\rangle _{{\mathcal{H}^{\Gamma}}', {\mathcal{H}^{\Gamma}}}=\int_0^S\Big(\int_{\Gamma(s)}\rho\eta dl \Big) ds=\int_0^S|\Gamma(s)|\hat{\rho}(s)\overline{\eta}(s)ds=\left\langle \hat{\rho},|\Gamma|\overline{\eta}\right\rangle_{\hat{V}',\hat{V}}\label{rho_eta}
\end{equation}
where we have used 
$\int_{\Gamma(s)}\rho~dl=|\Gamma(s)|\hat{\rho}$
and $\overline{\eta}(s)$ is the mean value of $\eta$ on the border $\Gamma(s)$ of each section. Exploiting \eqref{rho_eta} we can rewrite conditions (\ref{condpsi_u}) and (\ref{condpsi_hat}) as
\begin{align*}
&\left\langle \gamma_{_\Gamma}u-\psi,\eta\right\rangle_{\mathcal{H}^{\Gamma},{\mathcal{H}^{\Gamma}}'}=\left\langle |\Gamma|(\check{u}-\hat{\psi}),\overline{\eta}\right\rangle_{\hat{V},\hat{V}'}=0\\
&\left\langle \gamma_{_\Gamma}\tilde u-\psi,\eta\right\rangle_{\mathcal{H}^{\Gamma},{\mathcal{H}^{\Gamma}}'}=\left\langle |\Gamma|(\hat{u}-\hat{\psi}),\overline{\eta}\right\rangle_{\hat{V},\hat{V}'}=0
\end{align*} where $\check{u},\hat{\psi}\in \hat{V}$ are such that $\gamma_{_\Gamma}u=\extg\check{u}$, $\psi=\extg\hat{\psi}$ and $\gamma_{_\Gamma}\tilde{u}=\gamma_{_\Gamma}\exts\tilde{u}=\extg\hat{u}$, as $\tilde{u} \in \widetilde{V}$. 
Finally let us observe that
\begin{equation*}
(\bm{\tilde{K}}\nabla \tilde{u}, \nabla \tilde{v})_{\widetilde{V}}=\int_{\Sigma}\bm{\tilde{K}}\nabla \tilde{u}\nabla\tilde{v}~d\sigma=\int_0^S\bm{\tilde{K}}|\Sigma(s)|\cfrac{d\hat{u}}{ds}~\cfrac{d\hat{v}}{ds}~ds
\end{equation*}
where $\hat{u},\hat{v} \in \hat{V}$ are such that $\tilde{u}=\exts\hat{u}$, $\tilde{v}=\exts\hat{v}$ and $|\Sigma(s)|$ is the section area at $s \in [0,S]$. If we now extend the space $V_D$ from $D$ to the whole region $\Omega$, denoting this extended space by $V$, problem  (\ref{eq_var_u})-(\ref{condpsi_hat}) can be reduced to a 3D-1D coupled problem as:
 \textit{Find $(u,\hat{u}) \in V\times\hat{V}$, $\overline{\phi}\in \hat{V}'$ and $\hat{\psi} \in \hat{V}$ such that:}
\begin{align}
&(\bm{{K}}\nabla u, \nabla v)_{V}-\left\langle |\Gamma|\overline{\phi},\check{v} \right\rangle_{\hat{V}',\hat{V}}=(f,v)_{V}  \quad~\forall v \in V, \check{v} \in \hat{V}: \gamma_{_\Gamma} v=\mathcal{E}_{_\Gamma}\check{v} \label{equaz1} \\
&\Big( \bm{\tilde{K}}|\Sigma|\frac{d\hat{u}}{ds},\frac{d\hat{v}}{ds}\Big)_{\hat{V}}+\left\langle |\Gamma|\overline{\phi},\hat{v} \right\rangle_{\hat{V}',\hat{V}}=(|\Sigma|\overline{\overline{g}},\hat{v})_{\hat{V}}\quad~ \forall \hat{v} \in \hat{V}\label{equaz2}\\[0.5em]
&\qquad\left\langle |\Gamma|(\check{u}-\hat{\psi}),\overline{\eta}\right\rangle_{\hat{V}',\hat{V}}=0 \qquad \gamma_{_\Gamma}u=\mathcal{E}_{\Gamma}\check{u},~\forall \overline{\eta} \in \hat{V}'\label{condiz1}\\
&\qquad\left\langle |\Gamma|(\hat{u}-\hat{\psi}),\overline{\eta}\right\rangle_{\hat{V}',\hat{V}}=0 \qquad\forall \overline{\eta} \in \hat{V}'\label{condiz2}
\end{align}
with $\overline{\overline{g}}(s)=\frac{1}{|\Sigma(s)|}\int_{\Sigma(s)}g~d\sigma$, being $g$ sufficiently regular.

\section{PDE-constrained optimization problem}\label{PDE_constr}
The fulfillment of conditions (\ref{condiz1}) and (\ref{condiz2}) can be obtained through the minimization of a cost functional. Since we want to formulate independent problems on the various sub-domains, in order to guarantee the well posedness of each problem independently from the imposed boundary conditions, we modify equations (\ref{equaz1})-(\ref{equaz2}) as follows:
\begin{align}
&(\bm{{K}}\nabla u, \nabla v)_{V}+\alpha(|\Gamma|\check{u},\check{v} )_{\hat{V}}-\left\langle |\Gamma|\overline{\phi},\check{v} \right\rangle_{\hat{V}',\hat{V}} =(f,v)_V + \alpha(|\Gamma|\hat{\psi},\check{v} )_{\hat{V}} \label{eq_stab_u}\\[-0.4em]
&\hspace{7cm}\forall v \in V, ~\check{v} \in \hat{V}: \trg v=\extg\check{v}, \nonumber\\
&\Big(\bm{\tilde{K}}|\Sigma|\cfrac{d\hat{u}}{ds},\cfrac{d\hat{v}}{ds}\Big)_{\hat{V}}+\hat{\alpha}(|\Gamma|\hat{u},\hat{v})_{\hat{V}}+\left\langle |\Gamma|\overline{\phi},\hat{v} \right\rangle_{\hat{V}',\hat{V}}=(|\Sigma|\overline{\overline{g}},\hat{v})_{\hat{V}}+\hat{\alpha}(|\Gamma|\hat{\psi},\hat{v})_{\hat{V}}\label{eq_stab_hat}\\[-0.4em]
&\hspace{8cm}\qquad\forall \hat{v} \in \hat{V}\nonumber.
\end{align}
with $\alpha, \hat{\alpha} > 0$ being arbitrary parameters. Let us now define the following functional
\begin{eqnarray}
	J(\overline{\phi},\hat{ \psi})&=&\cfrac{1}{2}\left( ||\gamma_{_\Gamma}u(\overline{\phi},\hat{ \psi})-\psi||_{\hsp}^2+||\gamma_{_\Gamma}\tilde{u}(\overline{\phi},\hat{ \psi})-\psi||_{\hsp}^2\right) \nonumber \\
&=&\cfrac{1}{2}\left( ||\gamma_{_\Gamma}u(\overline{\phi},\hat{ \psi})-\extg \hat{\psi}||_{\hsp}^2+||\gamma_{_\Gamma}\exts \hat{u}(\overline{\phi},\hat{ \psi})-\extg \hat{\psi}||_{\hsp}^2\right)	\label{functional}
\end{eqnarray}
to be minimized constrained by \eqref{eq_stab_u} and \eqref{eq_stab_hat}. In order to rewrite the PDE-constrained optimization problem in a compact form, we consider the linear operators
$A: V \rightarrow V'$, $\widehat{A}:\hat{V} \rightarrow \hat{V}'$, $B:\hat{V}' \rightarrow V'$, $\widehat{B}:\hat{V}' \rightarrow \hat{V}'$, $C:\hat{V} \rightarrow V'$ and $\widehat{C}:\hat{V}\rightarrow \hat{V}'$ such that:
\begin{align}
&\left\langle Au,v\right\rangle_{V',V}=(\bm{{K}}\nabla u, \nabla v)_{V}+\alpha(|\Gamma|\check{u},\check{v})_{\hat{V}} \quad v \in V, ~\check{v} \in \hat{V}: \trg v=\extg\check{v}\label{A} \\
&\left\langle \widehat{A}\hat{u},\hat{v}\right\rangle_{\hat{V}',\hat{V}}=\Big(\bm{\tilde{K}}|\Sigma|\cfrac{d\hat{u}}{ds},\cfrac{d\hat{v}}{ds}\Big)_{\hat{V}}+\hat{\alpha}(|\Gamma|\hat{u},\hat{v} )_{\hat{V}}\qquad \qquad \hat{v} \in \hat{V} \label{Ahat}
\end{align}
\begin{align}
&\left\langle B \overline{\phi},v \right\rangle _{V',V}=\left\langle|\Gamma|\overline{\phi},\check{v} \right\rangle_{\hat{V}',\hat{V}} &~v \in V, ~\check{v} \in \hat{V}:\trg v=\extg\check{v} \label{B}\\
&\left\langle \widehat{B} \overline{\phi},\hat{v} \right\rangle _{\hat{V}',\hat{V}}=\left\langle |\Gamma|\overline{\phi},\hat{v} \right\rangle_{\hat{V}',\hat{V}} &~\hat{v} \in \hat{V} \label{Bhat}\\[0.5em]
&\left\langle C \hat{\psi},v \right\rangle _{V',V}=\alpha(|\Gamma|\hat{\psi},\check{v})_{\hat{V}} &~v \in V, ~\check{v} \in \hat{V} :\trg v=\extg\check{v} \label{C}\\
&\left\langle \widehat{C} \hat{\psi},\hat{v} \right\rangle _{\hat{V}',\hat{V}}=\hat{\alpha}(|\Gamma| \hat{\psi},\hat{v})_{\hat{V}} &~\hat{v} \in \hat{V}. \label{Chat}
\end{align}
The respective adjoints will be denoted as $A^*: V \rightarrow V'$, $\widehat{A}^*:\hat{V} \rightarrow \hat{V}'$, $B^*:V \rightarrow \hat{V}$, $\widehat{B}^*:\hat{V}\rightarrow \hat{V}$, $C^*:V \rightarrow \hat{V}'$, $\widehat{C}^*:\hat{V} \rightarrow \hat{V}'$. Let us further define 
\begin{align}
&F\in V' \text{ s.t. } F(v)=(f,v)_{V}, &v \in V\\
&G\in\hat{V}' \text{ s.t. }G(\hat{v})=(|\Sigma|\overline{\overline{g}},\hat{v})_{\hat{V}}, &\hat{v} \in \hat{V}.\label{G}
\end{align}
Equations \eqref{eq_stab_u}-\eqref{eq_stab_hat} can thus be written as: 
\begin{align}
&Au-B\overline{\phi}-C\hat{\psi}=F \label{eq1}\\
&\widehat{A}\hat{u}+\widehat{B}\overline{\phi}-\widehat{C}\hat{\psi}=G.\label{eq2}
\end{align}
If we now consider the space $\mathbb{V}=V \times \hat{V}$ and we set $\mathcal{W}=(u,\hat{u}) \in \mathbb{V}$ and $\mathcal{V}=(v,\hat{v})\in \mathbb{V}$, we can introduce the following operators:
\begin{align*}
&\mathcal{A}:\mathbb{V} \rightarrow \mathbb{V}' \text{ s.t. } \mathcal{A}(\mathcal{W},\mathcal{V})=A(u,v)+\widehat{A}(\tilde{u},\tilde{v})\\
&\mathcal{B}:\hat{V}' \rightarrow \mathbb{V}' \text{ s.t. } \mathcal{B}(\overline{\phi},\mathcal{V})=B(\overline{\phi},v)-\widehat{B}(\overline{\phi},\hat{v})\\
&\mathcal{C}:\hat{V} \rightarrow \mathbb{V}' \text{ s.t. } \mathcal{C}(\hat{\psi},\mathcal{V})=C(\hat{\psi},v)+\widehat{C}(\hat{\psi},\hat{v})
\end{align*}
and the PDE-constrained optimization problem can be written as
\begin{align}
&\min_{(\overline{\phi},\hat{\psi})}J(\overline{\phi},\hat{\psi}) \text{ subject to } \label{minJ}\\
&\quad \mathcal{A}\mathcal{W}-\mathcal{B}\overline{\phi}-\mathcal{C}\hat{\psi}=\mathcal{F}, \label{eq_compatta}
\end{align}
with $\mathcal{F}\in \mathbb{V}' \text{ s.t. } \mathcal{F}(\mathcal{V})=F(v)+G(\hat{v})$.
\begin{prop}
	Let us consider the trace operator $\gamma_{_\Gamma}:V\rightarrow \hsp$ and the extension operators $\mathcal{E}_{_\Sigma}:\hat{V} \rightarrow \tilde{V}\subset V$ and $\mathcal{E}_{_\Gamma}=\gamma_{_\Gamma}\circ \mathcal{E}_{_\Sigma}:\hat{V}\rightarrow \hsp$, whose respective adjoints are $\gamma_{_\Gamma}^*:{\hsp}'\rightarrow V'$, ${\exts}^*:\tilde{V}'\rightarrow \hat{V}'$ and ${\extg}^*:{\hsp}'\rightarrow \hat{V}'$ and let $\Theta_{\hat{V}}:\hat{V}\rightarrow\hat{V}'$ and $\Theta_{\hsp}:\hsp\rightarrow {\hsp}'$ be  Riesz isomorphisms. Then the optimal control $(\overline{\phi},\hat{\psi})$ that provides the solution to \eqref{minJ}-\eqref{eq_compatta} is such that
	\begin{align}
	 &\Theta_{\hat{V}}(B^*p-\widehat{B}^*\hat{p})=0\label{dphi}\\
	 &\Theta_{\hat{V}}^{-1}(C^*p+\widehat{C}^*\hat{p}-{\extg}^*\Theta_{\hsp}(\trg u(\overline{\phi},\hat{\psi})+\extg \hat{u}(\overline{\phi},\hat{\psi})-2\extg \hat{\psi}))=0\label{dpsi}
	 \end{align}
	 where $p \in V$ and $\hat{p} \in \hat{V}$ are the solutions respectively to
	 \begin{align}
	 &A^*p=\gamma_{_\Gamma}^*\Theta_{\hsp}(\gamma_{_\Gamma}u(\overline{\phi},\hat{\psi})-\extg\hat{\psi})\label{p}\\
	 &\widehat{A}^* \hat{p}={\extg}^*\Theta_{\hsp}(\extg\hat{u}(\overline{\phi},\hat{\psi})-\extg\hat{\psi})\label{phat}
	 \end{align}
	 \label{prop1}
\end{prop}
\begin{proof}
	Let us compute the Frechet derivatives of the functional with respect to the control variables $\overline{\phi}$ and $\hat{\psi}$. To this end, we introduce increments $\delta\overline{\phi} \in \hat{V}'$ of $\overline{\phi}$ and $\delta \hat{\psi} \in \hat{V}$ of $\hat{\psi}$ and we recall that $\exists~ \delta {\psi} \in \mathcal{H}^{\Gamma}~:~\delta \psi=\mathcal{E}_{_\Gamma}\delta\hat{ \psi}$. We have:
	\begin{align*}
	&\cfrac{\partial J}{\partial \overline{\phi}}\big(\overline{\phi}+\delta\overline{\phi},\hat{\psi}\big) =\left(\gamma_{_\Gamma}{u}(\overline{\phi},\hat{\psi})-\psi,\gamma_{_\Gamma}u(\delta\overline{\phi},0)\right)_{\hsp}+\left(\gamma_{_\Gamma}\tilde{u}(\overline{\phi},\hat{\psi})-\psi,\gamma_{_\Gamma}\tilde{u}(\delta\overline{\phi},0)\right)_{\hsp}=\\
	&=\left( \gamma_{_\Gamma}u(\overline{\phi},\hat{\psi})-\mathcal{E}_{_\Gamma}\hat{\psi},\gamma_{_\Gamma}u(\delta \overline{\phi},0)\right)_{\hsp}+\left(\trg \exts \hat{u}(\overline{\phi},\hat{\psi})-\extg \hat{\psi},\trg \exts \hat{u}(\delta\overline{\phi},0) \right)_{\hsp}=\\
	&=\left\langle \gamma_{_\Gamma}^*\Theta_{\hsp}(\gamma_{_\Gamma}u(\overline{\phi},\hat{\psi})-\extg\hat{\psi}),u(\delta \overline{\phi},0)\right\rangle_{V',V}+\left\langle{\extg}^*\Theta_{\hsp}(\extg\hat{u}(\overline{\phi},\hat{\psi})-\extg\hat{\psi}),\hat{u}(\delta\overline{\phi},0)\right\rangle_{\hat{V}',\hat{V}}=\\
	&=\left\langle A^*p,A^{-1}B\delta \overline{\phi} \right\rangle_{V',V}-\left\langle \widehat{A}^*\hat{p},\widehat{A}^{-1}\widehat{B}\delta \overline{\phi} \right\rangle_{\hat{V}',\hat{V}}=\\
	&=\left\langle B^*p,\delta \overline{\phi} \right\rangle_{{\hat{V}},\hat{V}'} -\left\langle \widehat{B}^*\hat{p},\delta \overline{\phi} \right\rangle_{{\hat{V}},\hat{V}'}=\left( \Theta_{\hat{V}}(B^*p-\widehat{B}^*\hat{p}),\delta \overline{\phi}\right)_{\hat{V}'};
	\end{align*}
	\begin{align*}
	\cfrac{\partial J}{\partial \hat{\psi}}(&\overline{\phi},\hat{\psi}+\delta \hat{\psi})=\left(\gamma_{_\Gamma}{u}(\overline{\phi},\hat{\psi})-\psi,\gamma_{_\Gamma}u(0,\delta\hat{\psi})-\delta\psi\right)_{\hsp}+\left(\gamma_{_\Gamma}\tilde{u}(\overline{\phi},\hat{\psi})-\psi,\gamma_{_\Gamma}\tilde{u}(0,\delta\hat{\psi})-\delta\psi\right)_{\hsp}=\\
	&=\left( \gamma_{_\Gamma}u(\overline{\phi},\hat{\psi})-\extg\hat{\psi},\gamma_{_\Gamma}u(0,\delta\hat{\psi})\right)_{\hsp}-\left( \gamma_{_\Gamma}u(\overline{\phi},\hat{\psi})-\extg\hat{\psi},\extg\delta\hat{\psi}\right)_{\hsp}+\\
	&\qquad+\left( \gamma_{_\Gamma}\exts \hat{u}(\overline{\phi},\hat{\psi})-\extg\hat{\psi},\gamma_{_\Gamma}\exts \hat{u}(0,\delta\hat{\psi})\right)_{\hsp}-\left( \gamma_{_\Gamma}\exts\hat{u}(\overline{\phi},\hat{\psi})-\extg\hat{\psi},\extg\delta\hat{\psi}\right)_{\hsp}=\\[0.5em]
	&=\left\langle A^*p,A^{-1}C\delta \hat{\psi} \right\rangle_{V',V}-\left\langle {\extg}^*\Theta_{\hsp}(\gamma_{_\Gamma}u(\overline{\phi},\hat{\psi})-\extg\hat{\psi}),\delta\hat{\psi}\right\rangle_{\hat{V}',\hat{V}}+\\
	&\qquad+\left\langle \widehat{A}^*\hat{p},\widehat{A}^{-1}\widehat{C}\delta \hat{\psi} \right\rangle_{\hat{V}',\hat{V}}-\left\langle {\extg}^*\Theta_{\hsp}(\extg\hat{u}(\overline{\phi},\hat{\psi})-\extg\hat{\psi}),\delta\hat{\psi}\right\rangle_{\hat{V}',\hat{V}}=\\[0.5em]
	&=\left\langle C^*p,\delta \hat{\psi} \right\rangle_{{\hat{V}}',\hat{V}} -\left\langle {\extg}^*\Theta_{\hsp}(\gamma_{_\Gamma}u(\delta\overline{\phi},\hat{\psi})-\extg\hat{\psi}),\delta\hat{\psi}\right\rangle_{\hat{V}',\hat{V}}+\\
	&\qquad\qquad +\left\langle \widehat{C}^*\hat{p},\delta \hat{\psi} \right\rangle_{{\hat{V}}',\hat{V}}-\left\langle {\extg}^*\Theta_{\hsp}(\extg\hat{u}(\overline{\phi},\hat{\psi})-\extg\hat{\psi}),\delta\hat{\psi}\right\rangle_{\hat{V}',\hat{V}}=\\
	&=\left( \Theta_{\hat{V}}^{-1}(C^*p+\widehat{C}^*\hat{p}-{\extg}^*\Theta_{\hsp}\trg u(\overline{\phi},\hat{\psi})-{\extg}^*\Theta_{\hsp}\extg \hat{u}(\overline{\phi},\hat{\psi})+2{\extg}^*\Theta_{\hsp}\extg \hat{\psi}),\delta \hat{\psi}\right)_{\hat{V}},
	\end{align*}
	which yield the thesis.
\end{proof}
\noindent Starting from the derivatives computed in Proposition \ref{prop1} let us define the quantities
\begin{align}
&\delta\overline{\phi}=\Theta_{\hat{V}}(B^*p-\widehat{B}^*\hat{p})\in \hat{V}'\\
&\delta \hat{\psi}=\Theta_{\hat{V}}^{-1}(C^*p+\widehat{C}^*\hat{p}-{\extg}^*\Theta_{\hsp}(\trg u(\overline{\phi},\hat{\psi})+\extg \hat{u}(\overline{\phi},\hat{\psi})-2\extg \hat{\psi})) \in \hat{V}.
\end{align}
Then the following proposition holds:
\begin{prop}
	Given the variable $\chi=(\overline{\phi},\hat{\psi})$, let us increment it by a step $\zeta \delta \chi$, where $\delta \mathcal{X}=(\delta \overline{\phi}, \delta \hat{\psi})$. The steepest descent method corresponds to the stepsize
	\begin{equation*}
	\zeta=-\cfrac{\left( \delta \overline{\phi},\delta \overline{\phi}\right)_{\hat{V}'}+\left( \delta \hat{\psi},\delta \hat{\psi}\right)_{\hat{V}}}{\left\langle B\delta \overline{\phi}+C\delta\hat{\psi},\delta p \right\rangle_{V',V}\hspace{-2mm}+\left\langle -\widehat{B}\delta\overline{\phi}+\widehat{C}\delta\hat{\psi},\delta \hat{p} \right\rangle_{\hat{V}',\hat{V}}\hspace{-2mm}-\left\langle{\extg}^*\Theta_{\hsp}(\trg \delta u+\extg \delta \hat{u}-2\extg \delta\hat{\psi}),\delta \hat{\psi}\right\rangle_{\hat{V}',\hat{V}}}
	\end{equation*}
	where
	\begin{align*}
	&\delta u=u(\delta \overline{\phi},\delta \hat{\psi})=A^{-1}(B\delta \overline{\phi}+C\delta \hat{\psi}) \in V,\\
	&\delta \hat{u}=\hat{u}(\delta \overline{\phi},\delta \hat{\psi} )=\widehat{A}^{-1}(-\widehat{B}\delta\overline{\phi}+\widehat{C}\delta\hat{\psi}) \in \hat{V}
	\end{align*}
	and $\delta p \in V$, $\delta \hat{p} \in \hat{V}$ are such that:
	\begin{align*}
 &A^*\delta p=\gamma_{_\Gamma}^*\Theta_{\hsp}(\gamma_{_\Gamma}\delta u-\extg\delta\hat{\psi})\\
&\widehat{A}^* \delta\hat{p}={\extg}^*\Theta_{\hsp}(\extg\delta \hat{u}-\extg\delta \hat{\psi})
	\end{align*}
\end{prop}
\begin{proof}
	It is sufficient to set to zero the derivative $\frac{\partial J(\chi+\zeta \delta \chi)}{\partial \zeta}$. In order to get a lighter notation let us set
	\begin{equation*}
	u=u(\overline{\phi},\hat{\psi});~~\delta u=u(\delta\overline{\phi},\delta\hat{\psi});\qquad
	\hat{u}=\hat{u}(\overline{\phi},\hat{\psi});~~\delta \hat{u}=\hat{u}(\delta\overline{\phi},\delta\hat{\psi});
	\end{equation*}
	\begin{align*}
		J&(\chi+\zeta\delta \chi)=J(\overline{\phi}+\zeta\delta \overline{\phi},\hat{\psi}+\zeta\delta \hat{\psi})=\\[0.5em]
		&=\cfrac{1}{2}\left(\gamma_{_\Gamma}u(\overline{\phi}+\zeta\delta \overline{\phi},\hat{\psi}+\zeta\delta \hat{\psi})-\psi-\zeta\delta\psi,\gamma_{_\Gamma}u(\overline{\phi}+\zeta \delta \overline{\phi},\hat{\psi}+\zeta\delta \hat{\psi})-\psi-\zeta\delta \psi \right)_{\hsp}\\[-0.2em]
		&+\cfrac{1}{2}\left(\gamma_{_\Gamma}\tilde{u}(\overline{\phi}+\zeta\delta \overline{\phi},\hat{\psi}+\zeta\delta \hat{\psi})-\psi-\zeta\delta\psi,\gamma_{_\Gamma}\tilde{u}(\overline{\phi}+\zeta \delta \overline{\phi},\hat{\psi}+\zeta\delta \hat{\psi})-\psi-\zeta\delta \psi \right)_{\hsp}=\\[0.5em]
		&=\cfrac{1}{2}\left(\gamma_{_\Gamma} u+\zeta\gamma_{_\Gamma}\delta u-\extg\hat{\psi}-\zeta\extg\delta\hat{\psi},\gamma_{_\Gamma}u+\zeta\gamma_{_\Gamma}\delta u-\extg\hat{\psi}-\zeta\extg\delta\hat{\psi}\right)_{\hsp}+\\[-0.2em]
		&+\cfrac{1}{2}\left(\gamma_{_\Gamma}\exts\hat{u}+\zeta\gamma_{_\Gamma}\exts\delta\hat{u}-\extg\hat{\psi}-\zeta\extg\delta\hat{\psi},\gamma_{_\Gamma}\extg\hat{u}+\zeta\gamma_{_\Gamma}\extg\delta\hat{u}-\extg\hat{\psi}-\zeta\extg\delta\hat{\psi}\right)_{\hsp}=\\[0.5em]	
		&=J(\overline{\phi},\hat{\psi})+\zeta\left( \gamma_{_\Gamma}u-\extg\hat{\psi},\gamma_{_\Gamma}\delta u-\extg\delta\hat{\psi}\right)_{\hsp}+\zeta\left( \extg\hat{u}-\extg\hat{\psi},\extg\delta\hat{u}-\extg\delta\hat{\psi}\right)_{\hsp}+\\[-0.2em]
		&+ \cfrac{\zeta^2}{2}\left( \gamma_{_\Gamma}\delta u-\extg\delta\hat{\psi},\gamma_{_\Gamma}\delta u-\extg\delta\hat{\psi}\right) _{\hsp}+\cfrac{\zeta^2}{2}\left( \extg\delta\hat{u}-\extg\delta\hat{\psi},\extg\delta \hat{u}-\extg\delta\hat{\psi}\right) _{\hsp}
	\end{align*}
	\begin{align*}
	&\cfrac{\partial J(\chi+\zeta \delta \chi)}{\partial \zeta}=\left( \gamma_{_\Gamma}u-\extg\hat{\psi},\gamma_{_\Gamma}\delta u-\extg\delta\hat{\psi}\right)_{\hsp}+\left( \extg\hat{u}-\extg\hat{\psi},\extg\delta\hat{u}-\extg\delta\hat{\psi}\right)_{\hsp}+\\[-0.2em]
	&\quad+\zeta\left( \gamma_{_\Gamma}\delta u-\extg\delta\hat{\psi},\gamma_{_\Gamma}\delta u-\extg\delta\hat{\psi}\right) _{\hsp}+\zeta\left( \extg\delta\hat{u}-\extg\delta\hat{\psi},\extg\delta\hat{u}-\extg\delta\hat{\psi}\right) _{\hsp}=0
	\end{align*}
	\begin{align*}
	\Rightarrow ~\zeta=-\cfrac{\left( \gamma_{_\Gamma}u-\extg\hat{\psi},\gamma_{_\Gamma}\delta u-\extg\delta\hat{\psi}\right)_{\hsp}+\left( \extg\hat{u}-\extg\hat{\psi},\extg\delta\hat{u}-\extg\delta\hat{\psi}\right)_{\hsp}}{\left( \gamma_{_\Gamma}\delta u-\extg\delta\hat{\psi},\gamma_{_\Gamma}\delta u-\extg\delta\hat{\psi}\right) _{\hsp}+\left( \extg\delta\hat{u}-\extg\delta\hat{\psi},\extg\delta\hat{u}-\extg\delta \hat{\psi}\right) _{\hsp}}
	\end{align*}
	Rearranging properly the terms we get \footnotesize
	\begin{align*}
	&\zeta=-\cfrac{\left\langle A^*p,A^{-1}(B\delta\overline{\phi}+C\delta\hat{\psi})\right\rangle _{V',V}\hspace{-2mm}+\left\langle\widehat{A}^*\hat{p},\widehat{A}^{-1}(-\widehat{B}\delta\overline{\phi}+\widehat{C}\delta\hat{\psi}) \right\rangle_{\hat{V}',\hat{V}}\hspace{-2mm}-\left\langle{\extg}^*\Theta_{\hsp}(\trg u+\extg \hat{u}-2\extg \hat{\psi}),\delta \hat{\psi}\right\rangle_{\hat{V}',\hat{V}}}{\hspace{-1mm}\left\langle A^{^{-1}}(B\delta\overline{\phi}+C\delta\hat{\psi}),A^{^*}\delta p\right\rangle _{V,V'}\hspace{-2mm}+\left\langle \widehat{A}^{^{-1}}\hspace{-0.8mm}(-\widehat{B}\delta\overline{\phi}+\widehat{C}\delta\hat{\psi}),\widehat{A}^*\delta \hat{p}\right\rangle_{\hat{V},\hat{V}'}\hspace{-2mm}-\left\langle{\extg}^*\Theta_{\hsp}(\trg \delta u+\extg \delta\hat{u}-2\extg \delta\hat{\psi},\delta \hat{\psi})\right\rangle_{\hat{V}',\hat{V}}}=\\[0.5em]
	&=-\cfrac{\left\langle B^*p,\delta \overline{\phi} \right\rangle_{{\hat{V}},\hat{V}'}\hspace{-2mm}+\left\langle C^*p,\delta \hat{\psi} \right\rangle_{{\hat{V}}',\hat{V}}\hspace{-2mm}-\left\langle \widehat{B}^*\hat{p},\delta \overline{\phi} \right\rangle_{{\hat{V}},\hat{V}'}\hspace{-2mm}+\left\langle \widehat{C}^*\hat{p},\delta \hat{\psi} \right\rangle_{{\hat{V}}',\hat{V}}\hspace{-2mm}-\left\langle{\extg}^*\Theta_{\hsp}(\trg u+\extg \hat{u}-2\extg \hat{\psi}),\delta \hat{\psi}\right\rangle_{\hat{V}',\hat{V}}}{\left\langle B\delta \overline{\phi}+C\delta\hat{\psi},\delta p \right\rangle_{V',V}\hspace{-2mm}+\left\langle -\widehat{B}\delta\overline{\phi}+\widehat{C}\delta\hat{\psi},\delta \hat{p} \right\rangle_{\hat{V}',\hat{V}}\hspace{-2mm}-\left\langle{\extg}^*\Theta_{\hsp}(\trg \delta u+\extg \delta \hat{u}-2\extg\delta \hat{\psi}),\delta \hat{\psi}\right\rangle_{\hat{V}',\hat{V}}}=\\[0.5em]
	&=-\cfrac{\left( \delta \overline{\phi},\delta \overline{\phi}\right)_{\hat{V}'}+\left( \delta \hat{\psi},\delta \hat{\psi}\right)_{\hat{V}}}{\left\langle B\delta \overline{\phi}+C\delta\hat{\psi},\delta p \right\rangle_{V',V}\hspace{-2mm}+\left\langle -\widehat{B}\delta\overline{\phi}+\widehat{C}\delta\hat{\psi},\delta \hat{p} \right\rangle_{\hat{V}',\hat{V}}\hspace{-2mm}-\left\langle{\extg}^*\Theta_{\hsp}(\trg \delta u+\extg \delta \hat{u}-2\extg \delta\hat{\psi}),\delta \hat{\psi}\right\rangle_{\hat{V}',\hat{V}}}
	\end{align*}\normalsize
	that yields the thesis.
\end{proof}

\section{Managing multiple cylinders and their intersections}\label{segment_inter_cont}
\newcommand{\I}{\mathcal{I}}
The previous discussion can be readily adapted to the case of multiple, say $\I$, small cylindrical inclusions $\Sigma_k$, with lateral surface $\Gamma_k$ and centreline $\Lambda_i$, $i=1,\ldots,\I$. Function spaces $\hat{V}$, $\widetilde{V}$, $\mathcal{H}^{\Gamma}$ are introduced on each segment and denoted by $\hat{V}_i$, $\widetilde{V}_i$, $\mathcal{H}^{\Gamma_i}$, while trace operator $\trg$, extension operator $\mathcal{E}_{_{\Gamma}}$  and $\exts$ are easily re-defined for each segment and denoted by $\trgi$, $\extgi$ and $\extsi$, respectively. Operators \eqref{A}-\eqref{Chat} are re-written as:
\begin{align*}
&\left\langle Au,v\right\rangle_{V',V}=(\bm{{K}}\nabla u, \nabla v)_{V}+\alpha\sum_{i=1}^{\I}(|\Gamma_i|\check{u}_i,\check{v}_i)_{\hat{V}_i} \\
&\hspace{2cm} \forall v \in V, ~\check{v}_i \in \hat{V}_i: \trgi v=\extgi\check{v}_i, \forall i=1,\ldots,\I  \\
&\left\langle \widehat{A}_i\hat{u}_i,\hat{v}_i\right\rangle_{\hat{V}_i',\hat{V}_i}=\Big(\bm{\tilde{K}}_i|\Sigma_i|\cfrac{d\hat{u}_i}{ds},\cfrac{d\hat{v}_i}{ds}\Big)_{\hat{V}_i}+\hat{\alpha}(|\Gamma_i|\hat{u}_i,\hat{v}_i )_{\hat{V}_i}\qquad \qquad \forall \hat{v}_i \in \hat{V}_i
\end{align*}
\begin{align*}
&\left\langle B_i \overline{\phi_i},v_i \right\rangle _{V_i',V_i}=\left\langle|\Gamma_i|\overline{\phi}_i,\check{v}_i \right\rangle_{\hat{V}_i',\hat{V}_i} &~\forall v \in V :\trgi v=\extgi\check{v}_i, ~\check{v}_i \in \hat{V}_i \\
&\left\langle \widehat{B}_i \overline{\phi}_i,\hat{v} \right\rangle _{\hat{V}_i',\hat{V}_i}=\left\langle |\Gamma_i|\overline{\phi}_i,\hat{v} \right\rangle_{\hat{V}_i',\hat{V}_i} &~\forall \hat{v} \in \hat{V}_i \\[0.5em]
&\left\langle C_i \hat{\psi}_i,v \right\rangle _{V',V}=\alpha(|\Gamma_i|\hat{\psi},\check{v})_{\hat{V}_i} &~ \forall v \in V, ~\check{v}_i \in \hat{V}_i :\trgi v=\extgi\check{v}\\
&\left\langle \widehat{C}_i \hat{\psi}_i,\hat{v} \right\rangle_{\hat{V}_i',\hat{V}_i}=\hat{\alpha}(|\Gamma_i| \hat{\psi}_i,\hat{v})_{\hat{V}_i} &~\hat{v} \in \hat{V}_i. 
\end{align*}
Problem \eqref{eq1}-\eqref{eq2} can be finally re-written for an arbitrary set of $\I$ centrelines as:
\begin{align}
&Au-B_i\overline{\phi}_i-C_i\hat{\psi}_i=F\label{eq1loc}\\
&\widehat{A}_i\hat{u}_i+\widehat{B}_i\overline{\phi}_i-\widehat{C}_i\hat{\psi}_i=G_i\label{eq2loc},
\end{align}
with the definition of $G_i$ following from the definition of an operator in the form of \eqref{G} for each segment. 
in which we consider different 1D variables $\hat{u}_i$, $\overline{\phi}_i$, $\hat{\psi}_i$ on the different 1D domains. In case cylindrical inclusions with intersecting centrelines, i.e. such that $\bar{\Lambda}_i\cup\bar{\Lambda}_j\neq\emptyset$, $i,j\in [1,\ldots,\I]$, we can still adopt formulation \eqref{eq1loc}-\eqref{eq2loc} by splitting the intersecting centrelines into non intersecting sub-segments and then adding continuity conditions at the intersection points.

Finally, the cost functional is re-written as:
 \begin{equation}
 J=\sum_{i=1}^\I J_i:=\sum_{k=1}^\I\cfrac{1}{2}\left( ||\gamma_{_{\Gamma_i}}u(\overline{\phi}_i,\hat{ \psi}_i)-\psi_i||_{\hsp}^2+||\gamma_{_\Gamma}\tilde{u}_i(\overline{\phi}_i,\hat{ \psi}_i)-\psi_i||_{\hsp}^2\right)
 \end{equation}
 where $\tilde{u}_i=\extsi\hat{u}_i$ and $\psi_i=\extgi\hat{\psi}_i$.

\section{Discrete matrix formulation}\label{Discrete}
Here the discrete matrix form of problem \eqref{minJ}-\eqref{eq_compatta} is presented.
The 3D-1D coupling is trivial in the discrete approximation spaces, given the regularity properties of the function spaces commonly used for discretization. Nonetheless, the present approach has the advantage of having a well posed mathematical formulation, and it even allows the use of non conforming meshes at the interfaces of the subdomains. Indeed, thanks to the optimization framework, it is possible to use completely independent meshes for the various domains and also for the interface variables, without any theoretical constraint on mesh sizes.

For the sake of generality we will consider, from the beginning, the presence of multiple segments crossing domain $\Omega$. 
Let us consider $\I$ segments of different length and orientation, defined ad $\Lambda_i=\left\lbrace \bm{\lambda}_i(s),s \in (0, S_i)\right\rbrace $, $i=1,...,\I$.
Let us consider a tetrahedral mesh $\mathcal{T}$ of domain $\Omega$, and let us define, on this mesh, Lagrangian finite element basis functions $\left\lbrace \varphi_k \right\rbrace_{k=1}^{N}$, such that $U=\sum_{k=1}^{N}U_k\varphi_k$ is the discrete approximation of pressure $u$. 
Let us then build three partitions of each segment $\Lambda_i$,  named $\hat{\mathcal{T}_i}$, $\tau^{\phi}_i$ and $\tau^{\psi}_i$, defined independently from each other and from $\mathcal{T}$. Let us further define the basis functions $\left\lbrace\hat{\varphi}_{i,k} \right\rbrace _{k=1}^{\hat{N}_i}$ on $\hat{\mathcal{T}_i}$, $\left\lbrace \theta_{i,k}\right\rbrace_{k=1}^{N_i^{\phi}}$ on $\tau^{\phi}_i$ and $\left\lbrace \eta_{i,k}\right\rbrace_{k=1}^{N_i^{\psi}}$ on $\tau^{\psi}_i$, with $\hat{N}_i$, $N_i^{\phi}$ and $N_i^{\psi}$ denoting the number of DOFs of the discrete approximations of the variables $\hat{u}_i$, $\overline{\phi}_i$ and $\hat{\psi}_i$ respectively, having set:
	\begin{equation*}
	\hat{U}_i=\sum_{k=1}^{\hat{N}_i}\hat{U}_{i,k}\hat{\varphi}_{i,k}, \quad \Phi_i=\sum_{k=1}^{N_i^{\phi}}\Phi_{i,k}\theta_{i,k}, \quad \Psi_i=\sum_{k=1}^{N_i^{\psi}}\Psi_{i,k}\eta_{i,k}
	\end{equation*}

We then define the following matrices:
\begin{align*}
&\bm{A} \in \mathbb{R}^{N\times N} \text{ s.t. } (A)_{kl}=\int_{\Omega}\bm{{K}}\nabla\varphi_k\nabla\varphi_l ~d\omega+\alpha\sum_{i=1}^{\I}\int_{\Lambda_i}|\Gamma(s_i)|{\varphi_k}_{|_{\Lambda_i}}{\varphi_l}_{|_{\Lambda_i}} ds\\[0.8em]
&\bm{\hat{A}_i} \in \mathbb{R}^{\hat{N}_i\times \hat{N}_i} \text{ s.t. } (\hat{A}_i)_{kl}=\int_{\Lambda_i}\bm{\tilde{K}_i}|\Sigma(s_i)|\frac{d\hat{\varphi}_{i,k}}{ds}\frac{d\hat{\varphi}_{i,l}}{ds} ~ds+\hat{\alpha}\int_{\Lambda_i}|\Gamma(s_i)|\hat{\varphi}_{i,k}\hat{\varphi}_{i,l}~ds
\end{align*}
\begin{align*}
&\bm{B_i} \in \mathbb{R}^{N\times N_i^{\phi}} \text{ s.t. } (B_i)_{kl}=\int_{\Lambda_i}|\Gamma(s_i)|{{\varphi_{k}}_{|_{\Lambda_i}}\theta_{i,l}}~ds\\
&\bm{\hat{B}_i} \in \mathbb{R}^{\hat{N}_i\times N_i^{\phi}} \text{ s.t. } (\hat{B}_i)_{kl}=\int_{\Lambda_i}|\Gamma(s_i)|{\hat{\varphi}_{i,k}~\theta_{i,l}}~ds\\
&\bm{C_i}^{\alpha} \in \mathbb{R}^{N\times N_i^{\psi}} \text{ s.t. } (C_i^{\alpha})_{kl}=\alpha\int_{\Lambda_i}|\Gamma(s_i)|{\varphi_k}_{|_{\Lambda_i}}\eta_{i,l}~ds\\
&\bm{\hat{C}_i}^{\alpha} \in \mathbb{R}^{\hat{N}_i\times N_i^{\psi}} \text{ s.t. } (\hat{C_i}^{\alpha})_{kl}=\hat{\alpha}\int_{\Lambda_i}|\Gamma(s_i)|{\hat{\varphi}}_{i,k}~\eta_{i,l}~ds,
\end{align*}
and the vectors 
\begin{align*}
&f\in \mathbb{R}^N \text{ s.t. } f_k=\int_{\Omega}f\varphi_k~d\omega\\
&g_i\in \mathbb{R}^{\hat{N}_i} \text{ s.t. } (g_i)_k=\int_{\Lambda_i}|\Sigma(s_i)|\overline{\overline{g}}~\hat{\varphi}_{i,k}~ds.
\end{align*}
Setting $\hat{N}=\sum_{i=1}^{\I}\hat{N}_i$, $N^{\psi}=\sum_{i=1}^{\I}N_i^{\psi}$ and $N^{\phi}=\sum_{i=1}^{\I}\ N_i^{\phi}$, 
we can group the matrices as follows for all the segments in the domain:
\begin{align*}
&\bm{B}=\left[ \bm{B_1}, \bm{B_2},...,\bm{B_{\I}}\right]  \in \mathbb{R}^{N\times N^{\phi}} \qquad \bm{\hat{B}}=\text{diag}\left( \bm{\hat{B}_1},...,\bm{\hat{B}_{\I}}\right)   \in \mathbb{R}^{\hat{N}\times N^{\phi}}\\
&\bm{C}^{\alpha}=\left[ \bm{C_1}^{\alpha}, \bm{C_2}^{\alpha},...,\bm{C_{\I}}^{\alpha}\right] \in \mathbb{R}^{N\times N^{\psi}} \qquad \bm{\hat{C}}^{\alpha}=diag\left(  \bm{\hat{C}_1}^{\alpha},...,\bm{\hat{C}_{\I}}^{\alpha}\right)   \in \mathbb{R}^{\hat{N}\times N^{\psi}}
\end{align*}
and, for non intersecting segments, we have:
\begin{equation*}
\bm{\hat{A}}=\text{diag}\left(  \bm{\hat{A}_1},...,\bm{\hat{A}_{\I}}\right)   \in \mathbb{R}^{\hat{N}\times\hat{N}},
\end{equation*}
whereas, for groups of intersecting segments, we proceed as described in Section~\ref{segment_inter_cont} and we enforce continuity through Lagrange multipliers. For each connected group of segments we thus have:
\begin{displaymath}
\bm{\hat{A}^\star_\zeta}=\left[
\begin{array}{cc}
\text{diag}\left(  \bm{\hat{A}_{\zeta_1}},...,\bm{\hat{A}_{\zeta_n}}\right) & \bm{Q}^T\\
\bm{Q} & \bm{0}
\end{array}
\right]
\end{displaymath}
where matrix $\bm{Q}$ simply equates the DOFs at the extrema of connected sub-segments. Matrices $\bm{\hat{A}^\star_\zeta}$, for $\zeta$ spanning the whole number of connected groups of segments, are assembled block diagonally to form matrix $\bm{\hat{A}}$. Please note that, for disconnected segments, each matrix $\bm{\hat{A}^\star_\zeta}$ coincides with matrix $\bm{\hat{A}_\zeta}$. 
Finally we can write
\begin{align}
&\bm{A}U-\bm{B}\Phi-\bm{C}^{\alpha}\Psi=f\label{eq1discr}\\
&\bm{\hat{A}}\hat{U}+\bm{\hat{B}}\Phi-\bm{\hat{C}}^{\alpha}\Psi=g\label{eq2discr}
\end{align}
with
\begin{align*}
&\hat{U}=\left[\hat{U}_1^T,...,\hat{U}_{\I}^T \right]^T \in \mathbb{R}^{\hat{N}}; \quad g=[g_1^T,g_2^T,...,g_{\I}^T]^T\in\mathbb{R}^{\hat{N}}\\ 
&\Phi=\left[\Phi_1^T,...,\Phi_{\I}^T \right]^T \in \mathbb{R}^{N^{\phi}}; \quad \Psi=\left[\Psi_1^T,...,\Psi_{\I}^T \right]^T\in \mathbb{R}^{N^{\psi}}.
\end{align*}

In order to get a more compact form of the previous equations, let us set $W=(U,\hat{U})$ and
\begin{equation}
\bm{\mathcal{A}}=\begin{bmatrix}
\bm{A} & 0\\
0 &\bm{\hat{A}}
\end{bmatrix}, \qquad
\bm{\mathcal{B}}=\begin{bmatrix}
\bm{B}\\
-\bm{\hat{B}}
\end{bmatrix},\qquad
\bm{\mathcal{C}}^{\alpha}=\begin{bmatrix}
\bm{C}^{\alpha}\\
\bm{\hat{C}}^{\alpha}
\end{bmatrix}\qquad
\mathcal{F}=\begin{bmatrix}
f\\
g
\end{bmatrix},
\label{Adef}
\end{equation}
so that the discrete constraint equations become:
\begin{equation}
\bm{\mathcal{A}}W-\bm{\mathcal{B}}\Phi+\bm{\mathcal{C}}^{\alpha}\Psi=\mathcal{F}.\label{eqdiscr_compatta}
\end{equation}

Concerning the cost functional in \eqref{functional}, first we define matrices
\begin{align*}
&\bm{G_i} \in \mathbb{R}^{N \times N} \text{ s.t. } (G_i)_{kl}=\int_{\Lambda_i}{\varphi_k}_{|_{\Lambda_i}}{\varphi_l}_{|_{\Lambda_i}}ds\\
&\bm{\hat{G}_i} \in \mathbb{R}^{\hat{N}_i \times \hat{N}_i} \text{ s.t. } (\hat{G}_i)_{kl}=\int_{\Lambda_i}{\hat{\varphi}}_{i,k}~{\hat{\varphi}}_{i,l}~ds\\
&\bm{G_i^{\psi}} \in \mathbb{R}^{N_i^{\psi} \times N_i^{\psi}}  \text{ s.t. } (G_i^{\psi})_{kl}=\int_{\Lambda_i}\eta_{i,k}~\eta_{i,l}~ds \\
&\bm{C_i} \in \mathbb{R}^{N\times N_i^{\psi}} \text{ s.t. } (C_i)_{kl}=\int_{\Lambda_i}{\varphi_k}_{|_{\Lambda_i}}\eta_{i,l}~ds\\
&\bm{\hat{C}_i} \in \mathbb{R}^{\hat{N}_i\times N_i^{\psi}} \text{ s.t. } (\hat{C_i})_{kl}=\int_{\Lambda_i}{\hat{\varphi}}_{i,k}~\eta_{i,l}~ds
\end{align*}
and then
\begin{equation}
\label{Gdef}
\bm{G}=\sum_{i=1}^{\I}\bm{G}_i \in \mathbb{R}^{N \times N} \qquad \bm{\hat{G}}=\text{diag}\left( \bm{\hat{G}_1}^T,...,\bm{\hat{G}_{\I}}^T\right)  \in \mathbb{R}^{\hat{N} \times\hat{N}} \qquad \bm{\mathcal{G}}=\begin{bmatrix}
\bm{G} & 0\\0 &\bm{\hat{G}}
\end{bmatrix}
\end{equation}
\begin{equation*}
\bm{G^{\psi}}=\text{diag}\left(\bm{G_1^{\psi}},...,\bm{G_{\I}^{\psi}} \right) \in \mathbb{R}^{N^{\psi}\times N^{\psi}}
\end{equation*}
\begin{align*}
\bm{C}=\left[ \bm{C_1}, \bm{C_2},...,\bm{C_{\I}}\right] \in \mathbb{R}^{N\times N^{\psi}} \quad \bm{\hat{C}}=\text{diag}\left(  \bm{\hat{C}_1},...,\bm{\hat{C}_{\I}}\right)   \in \mathbb{R}^{\hat{N}\times N^{\psi}} \quad \bm{\mathcal{C}}=\begin{bmatrix} \bm{C} \\ \bm{\hat{C}} \end{bmatrix}.
\end{align*}
The discrete cost functional then reads:
\begin{align}
\tilde{J}&=\cfrac{1}{2}\left( U^T\bm{G}U-U^T\bm{C}\Psi-\Psi^T\bm{C}^TU+\hat{U}^T\bm{\hat{G}}\hat{U}-\hat{U}^T\bm{\hat{C}}\Psi-\Psi^T\bm{\hat{C}}^T\hat{U}+2\Psi^T\bm{G^{\psi}}\Psi\right)=\nonumber \\
&=\cfrac{1}{2}\left( W^T\bm{\mathcal{G}}W-W^T\bm{\mathcal{C}}\Psi-\Psi^T\bm{\mathcal{C}}^TW+2\Psi^T\bm{G^{\psi}}\Psi\right). \label{Jtilde}
\end{align}
The discrete matrix formulation of the 3D-1D problem finally takes the form:
\begin{align}
	\min_{(\Phi,\Psi)}\tilde{J}(\Phi,\Psi) \text{ subject to } (\ref{eqdiscr_compatta}). \label{minJtilde}
\end{align}
First order optimality conditions for the above problem correspond to the saddle-point system:

\begin{equation}
\label{KKT}
\bm{\mathcal{S}}=\begin{bmatrix}
\bm{\mathcal{G}} & \bm{0} & -\bm{\mathcal{C}} & \bm{\mathcal{A}}^T \\ 
\bm{0} & \bm{0} & \bm{0} & \bm{\mathcal{B}}^T \\
-\bm{\mathcal{C}}^T & \bm{0} & 2\bm{G^{\psi}} & (-\bm{\mathcal{C}}^{\alpha})^T\\
\bm{\mathcal{A}} & \bm{\mathcal{B}} & -\bm{\mathcal{C}}^{\alpha} & \bm{0}\\
\end{bmatrix}
\end{equation}
\begin{equation}
\label{optprob}
\bm{\mathcal{S}}\begin{bmatrix}
W\\ \Phi \\ \Psi \\ -P
\end{bmatrix}=\begin{bmatrix}
\mathcal{F} \\ \bm{0}  \\ \bm{0} \\ \bm{0} 
\end{bmatrix}
\end{equation}
\begin{prop}
\label{discretewellposedness}
Matrix $\bm{\mathcal{S}}$ in \eqref{KKT} is non-singular and the unique solution of \eqref{optprob} is equivalent to the solution of the optimization problem  \eqref{minJtilde}.
\end{prop}
The proof of Proposition \ref{discretewellposedness} derives from classical arguments of quadratic programming once the following lemma is proven:
\begin{lemma}
\label{lemma1}

Let matrix $\bm{\mathcal{A}^\star}$ be as
\begin{displaymath}
\bm{\mathcal{A}^\star}=
\begin{bmatrix}
\bm{\mathcal{A}} & \bm{\mathcal{B}} & -\bm{\mathcal{C}}^\alpha
\end{bmatrix}
\end{displaymath}
and let $\bm{\mathcal{G}^\star}$ be defined as
\begin{displaymath}
\bm{\mathcal{G}^\star}=
\begin{bmatrix}
\bm{\mathcal{G}} & \bm{0} & -\bm{\mathcal{C}} \\ 
\bm{0} & \bm{0} & \bm{0} \\
-\bm{\mathcal{C}}^T & \bm{0} & 2\bm{G^{\psi}}
\end{bmatrix}
\end{displaymath}
Then matrix $\bm{\mathcal{A}^\star}$ is full rank and matrix $\bm{\mathcal{G}^\star}$ is symmetric positive definite on $\ker(\bm{\mathcal{A}^\star})$.

\end{lemma}
\begin{proof}
The proof is adapted from the one provided in \cite{NostroArxiv}, we report here the key steps.
Matrix $\bm{\mathcal{A}}$ is full rank and
matrix $\bm{\mathcal{G}^\star}$ is symmetric positive semi-definite by construction. 
We thus need to show that $\ker{(\bm{\mathcal{G}^\star})} \cap \ker{(\bm{\mathcal{A}^\star})} =\{0 \}$.
Let us consider the canonical basis for $\mathbb{R}^{N^{\Phi}+N^{\Psi}}$ and let $e_k$ denote  the $k$-th element of such basis,  $k=1,\ldots,N^\Phi+N^\Psi$. Let $z_k\in\ker{(\bm{\mathcal{A}^\star})}$ be defined as:
\begin{displaymath}
z_k=\begin{bmatrix}
\bm{\mathcal{A}}^{-1}\begin{bmatrix}
\bm{\mathcal{B}} & -\bm{\mathcal{C}}^\alpha
\end{bmatrix}
e_k\\
e_k
\end{bmatrix}.
\end{displaymath}
Let us assume that $1\leq k\leq N^{\Phi}$, thus corresponding to a non null value of the variable $\Phi$ on one segment. This in turn  gives a non null value $U$ and $\hat{U}$ on the traces and thus a non null value of the functional, or $z_k^T \bm{\mathcal{G}^\star}z_k>0$.
If instead $N^{\Phi}+1\leq k\leq N^{\Psi}$, this corresponds to a non-null variable $\Psi:=e_k$. Since the solution to \eqref{optprob} is the same for every value of $\alpha$ and $\hat{\alpha}$, included $\alpha=\hat{\alpha}=0$ (the consistent terms depending on $\alpha$ and $\hat{\alpha}$ are only required for the independent resolution on the sub-domains), we choose here $\alpha=\hat{\alpha}=0$, so that:  
$$z_k=\begin{bmatrix}
\bm{\mathcal{A}}^{-1}\begin{bmatrix}
\bm{\mathcal{B}} & \bm{0}
\end{bmatrix}
e_k\\
e_k
\end{bmatrix}:=\begin{bmatrix} \bm{0}\\e_k \end{bmatrix}$$
thus $U$, $\hat{U}$ are null, and therefore we can conclude that $z_k^T \bm{\mathcal{G}}z_k>0$ also in this case (see \cite{NostroArxiv} for the proof with $\alpha,\hat{\alpha}>0$).
Summarizing we have shown that  $z_k\not\in\ker{(\bm{\mathcal{G}}^\star)}$ for any $k=1,\ldots, N^{\Phi}+N^{\Psi}$. The vector space $\ker{(\bm{\mathcal{A}}^\star)}=\text{span}\{z_1,\ldots,z^{N^{\Phi}+N^{\Psi}}\}$ is a subspace of $\text{Im}(\bm{\mathcal{G}}^\star)$, and $\ker{(\bm{\mathcal{G}}^\star)} \cap \ker{(\bm{\mathcal{A}}^\star)} =\{\bm{0} \}$.
\end{proof}

System \eqref{optprob} can be used to obtain a numerical solution to the problem. For very large problems instead, the above system is likely to become ill-conditioned, and thus an alternative resolution strategy is proposed, as described below.
By formally replacing
$W=\bm{\mathcal{A}}^{-1}(\bm{\mathcal{B}}\Phi-\bm{\mathcal{C}}^{\alpha}\Psi+\mathcal{F})$
in the functional \eqref{Jtilde}, we obtain
\begin{align}
J^\star(\Phi,\Psi)&=\cfrac{1}{2}\Big( (\bm{\mathcal{A}}^{-1}\bm{\mathcal{B}}\Phi+\bm{\mathcal{A}}^{-1}\bm{\mathcal{C}}^{\alpha}\Psi+\bm{\mathcal{A}}^{-1}\mathcal{F})^T\bm{\mathcal{G}}(\bm{\mathcal{A}}^{-1}\bm{\mathcal{B}}\Phi+\bm{\mathcal{A}}^{-1}\bm{\mathcal{C}}^{\alpha}\Psi+\bm{\mathcal{A}}^{-1}\mathcal{F})+\nonumber\\
&\qquad-(\bm{\mathcal{A}}^{-1}\bm{\mathcal{B}}\Phi+\bm{\mathcal{A}}^{-1}\bm{\mathcal{C}}^{\alpha}\Psi+\bm{\mathcal{A}}^{-1}\mathcal{F})^T\bm{\mathcal{C}}\Psi+\nonumber\\&\qquad-\Psi^T\bm{\mathcal{C}}^T(\bm{\mathcal{A}}^{-1}\bm{\mathcal{B}}\Phi+\bm{\mathcal{A}}^{-1}\bm{\mathcal{C}}^{\alpha}\Psi+\bm{\mathcal{A}}^{-1}\mathcal{F})\Big)=\nonumber\\
&=\cfrac{1}{2}~[\Phi^T \quad \Psi^T]\begin{bmatrix}
\bm{\mathcal{B}}^T\bm{\mathcal{A}}^{-T}\bm{\mathcal{G}}\bm{\mathcal{A}}^{-1}\bm{\mathcal{B}} &\quad \Large\substack{\bm{\mathcal{B}}^T\bm{\mathcal{A}}^{-T}\bm{\mathcal{G}}\bm{\mathcal{A}}^{-1}\bm{\mathcal{C}}^{\alpha}+\\-\bm{\mathcal{B}}^T\bm{\mathcal{A}}^{-T}\bm{\mathcal{C}}}\\\\
\Large\substack{(\bm{\mathcal{C}}^{\alpha})^T\bm{\mathcal{A}}^{-T}\bm{\mathcal{G}}\bm{\mathcal{A}}^{-1}\bm{\mathcal{B}}+\\-\bm{\mathcal{C}}^T\bm{\mathcal{A}}^{-1}\bm{\mathcal{B}}} & \Large\substack{(\bm{\mathcal{C}}^{\alpha})^T\bm{\mathcal{A}}^{-T}\bm{\mathcal{G}}\bm{\mathcal{A}}^{-1}\bm{\mathcal{C}}^{\alpha}+\nonumber \\-\bm{\mathcal{C}}^T\bm{\mathcal{A}}^{-T}\bm{\mathcal{C}}^{\alpha}+\\-(\bm{\mathcal{C}}^{\alpha})^T\bm{A}^{-1}\bm{\mathcal{C}}+2\bm{G^{\psi}}}
\end{bmatrix}\begin{bmatrix}
\Phi\smallskip\\\\\\\Psi
\end{bmatrix}+\\[0.5em]
&\qquad+\mathcal{F}^T\begin{bmatrix}
\bm{\mathcal{A}}^{-T}\bm{\mathcal{G}}\bm{\mathcal{A}}^{-1}\bm{\mathcal{B}} &\quad \bm{\mathcal{A}}^{-T}\bm{\mathcal{G}}\bm{\mathcal{A}}^{-1}\bm{\mathcal{C}}^{\alpha}-\bm{\mathcal{A}}^{-T}\bm{\mathcal{C}}
\end{bmatrix}\begin{bmatrix}
\Phi\\\Psi
\end{bmatrix}+\nonumber \\
&\qquad +\cfrac{1}{2}\left( \mathcal{F}^T\bm{\mathcal{A}}^{-T}\bm{\mathcal{G}}\bm{\mathcal{A}}^{-1}\mathcal{F}\right)= \nonumber \\ 
&=\cfrac{1}{2}\left( \mathcal{X}^T\bm{M}\mathcal{X}+2d^T\mathcal{X}+q\right).\label{Jcompact}
\end{align}
Matrix $\bm{M}$ is symmetric positive definite as follows from the equivalence of this formulation with the previous saddle point system \eqref{optprob}, and thus the minimization of the unconstrained problem \eqref{Jcompact} can be performed via a gradient based scheme. It is however to remark that the computation of gradient direction at point $\mathcal{X}^\sharp$, i.e. $\nabla J^\star(\mathcal{X}^\sharp)=\bm{M}\mathcal{X}^\sharp+d$, can be performed in a matrix free format and involves the independent factorization of the 1D matrices $\bm{\hat{A}_i}$, $i=1,\ldots,L$ and of the 3D elliptic matrix $\bm{A}$, which are all non singular as long as $\alpha, \hat{\alpha}>0$. The analysis of this solving strategy and of its potential for parallel computing is deferred to a forthcoming work.

\section{Numerical results}\label{Num_res}

\begin{figure}
	\centering
	\includegraphics[width=0.80\linewidth]{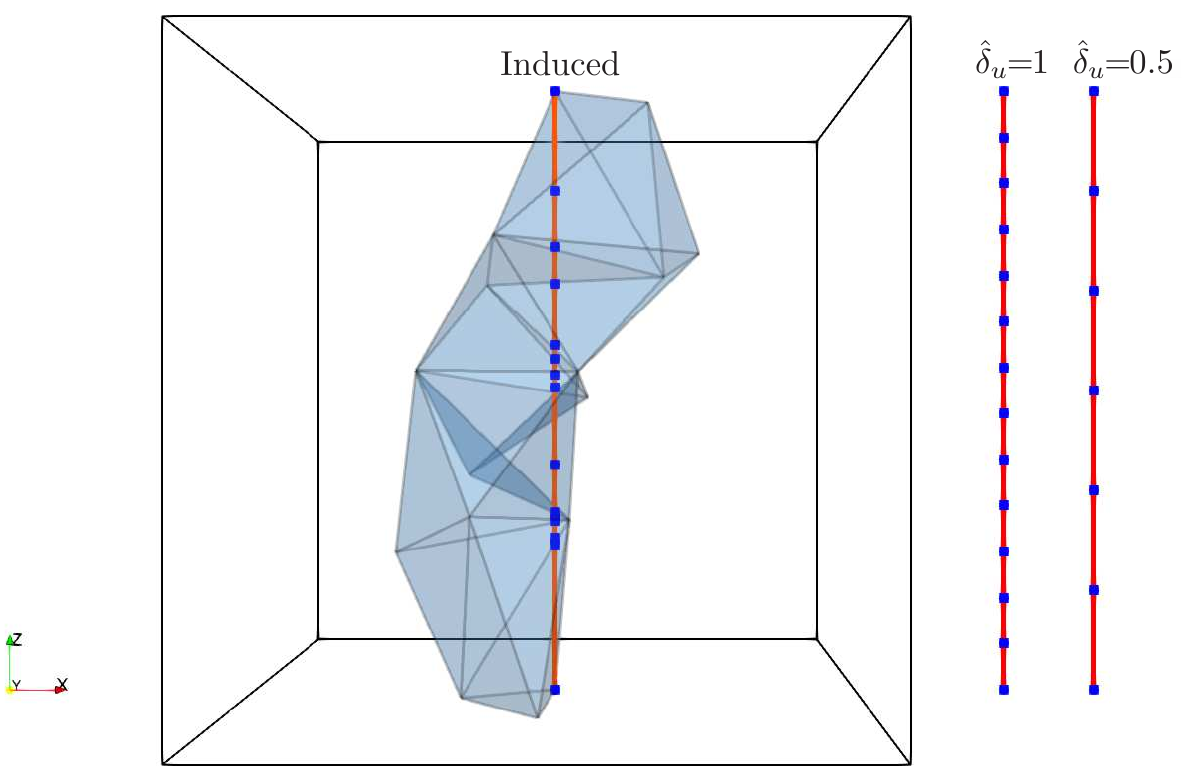}
	\caption{Highlight on the tetrahedra intersected by a segment and consequently induced mesh; on the right equispaced partitions of the segment for $\hat{\delta}_u=1$ and for $\hat{\delta}_u=0.5$.}
	\label{fig:meshes}
\end{figure}

In this section we propose some numerical test to validate the proposed approach and to show its applicability to the problem of interest. Three numerical tests are proposed. A first problem called  \textit{Test Problem 1} (TP1) takes into account a single cylindrical inclusion and has a smooth analytical solution, thus allowing to evaluate convergence trends for the error. The second test, called \textit{Test Problem 2} (TP2),  takes into account a different problem with no known analytical solution on a similar geometry. In this case, the obtained solution is compared to a 3D-3D simulation with standard conforming finite elements, and different values of the coefficient $\bm{\tilde{K}}$ are considered. Finally an example with multiple intersecting inclusions is proposed, to test the behavior of the method in more general settings. In this case, a qualitative evaluation on the behavior of the numerical solution is proposed, along with a quantitative evaluation of a proposed error indicator.

All the simulations are performed using finite elements on 3D and 1D non-conforming meshes, independently generated on the sub-domains. A mesh parameter $h$ is used to denote the maximum diameter of the tetrahedra for the 3D mesh of $\Omega$, whereas the refinement level of the 1D meshes $\hat{\mathcal{T}_i}$, $\tau^{\phi}_i$, $\tau^{\psi}_i$, $i=1,...\I$, is provided in relation to the mesh-size of the 1D mesh induced on the segments $\Lambda_i$ by the tetrahedral mesh, i.e. the 1D mesh given by the intersections of $\Lambda_i$ with the tetrahedra in $\mathcal{T}$, see Figure~\ref{fig:meshes}. This is done in order to better highlight the relative sizes of the various meshes. In particular the adimensional number $\hat{\delta}_{u,i}$ denotes the ratio between the number of elements of the mesh in $\hat{\mathcal{T}_i}$ with respect to the number of elements of the induced mesh on $\Lambda_i$, whereas $\delta_{\phi,i}$ and $\delta_{\psi,i}$ the ratio between the number of elements in $\tau^{\phi}_i$, $\tau^{\psi}_i$, respectively, and the induced mesh. Figure~\ref{fig:meshes} shows the induced mesh, in the middle, and two 1D meshes, e.g. for $\hat{\mathcal{T}}$ corresponding to values of $\hat{\delta}_{u}=1$ and $\hat{\delta}_{u}=0.5$ and equally-spaced nodes. In the simulations, for simplicity, we will always use equally-spaced nodes for the 1D meshes and unique different values of $\hat{\delta}_{u}$, $\delta_{\phi}$ and $\delta_{\psi}$ for the various segments, thus dropping, in the following, the reference to segment index for these parameters. Linear Lagrangian finite elements are used on $\mathcal{T}$, $\hat{\mathcal{T}_i}$, $\tau^{\psi}_i$, whereas piece-wise constant basis functions are used to describe variables $\phi_i$ on $\tau^{\phi}_i$,  $i=1,...\I$. All numerical tests are performed using $\alpha=\hat{\alpha}=1$, even if it is to remark that the value of such parameters has no impact on the solution, and, as long as formulation \eqref{optprob} is used, $\alpha=\hat{\alpha}=0$ could have been also chosen. 

\subsection[Test1]{Test Problem 1 (TP1)}
Let us consider a cube $\Omega$ of edge $l$ inscribed in a cylinder of radius $\hat{R}=\frac{l\sqrt{2}}{2}$ centered in the axis origin and a cylinder $\Sigma$ of radius $\check{R}<\hat{R}$ and height $h=l$ whose centreline $\Lambda$ lies on the $z$ axis (see Figure~\ref{section_view}). Let us denote by $\partial \Omega_{l}$, $\partial \Omega_{+}$ and $\partial \Omega_{-}$ respectively the lateral, the top and the bottom faces of the cube.
\begin{figure}
	\centering
	\includegraphics[width=0.3\linewidth]{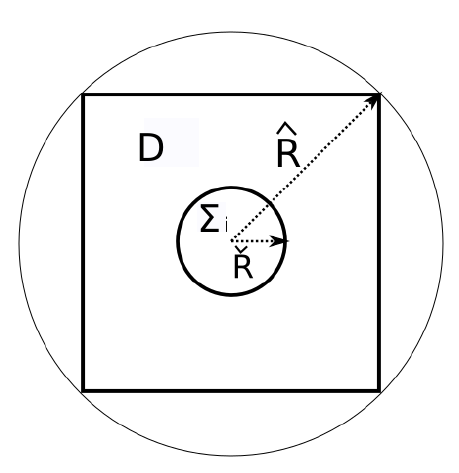}\hspace{10mm}%
	\includegraphics[width=0.45\linewidth]{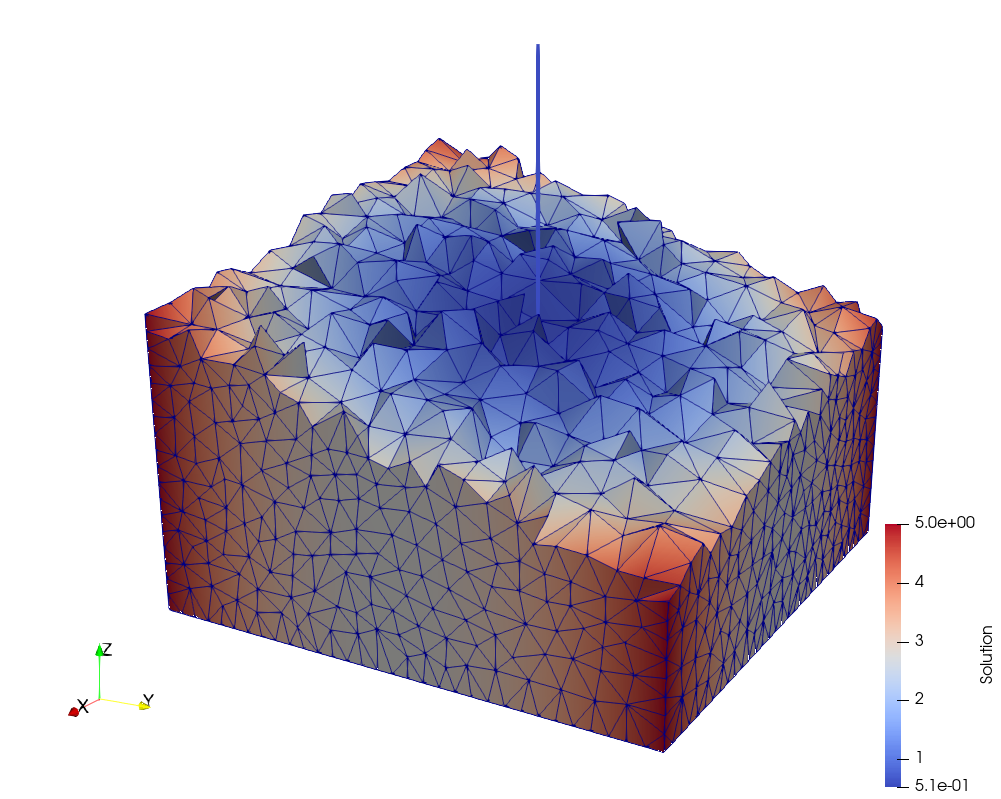}
	\caption{TP1: on the left, section view of the starting 3D-3D domain for the TestProblem1 experiment; on the right solution obtained inside the cube and on the segment for $h=0.086$, $\hat{\delta}_u=1$, $\delta_{\phi}=0.5$ and $\delta_{\psi}=0.5$.}
	\label{section_view}
\end{figure}
Given $a,b,c,k_1,k_2 \in \mathbb{R}$, let us consider a problem in the form of \eqref{equaz1}-\eqref{condiz2}, obtained by reducing $\Sigma$ to its centerline, with $K=\tilde{K}=1$ and
\begin{align*}
f=-\cfrac{b}{\sqrt{x^2+y^2}}-4a, \qquad \overline{\overline{g}}=0.
\end{align*}

The problem is completed with the appropriate boundary condition to have the exact solution given by:
\begin{align}
&u_{ex}(x,y,z)=a(x^2+y^2)+b\sqrt{x^2+y^2}+c \label{uestest1} & \text{ in } \Omega\\
&\hat{u}_{ex}(x,y,z)=k_1&\label{uhatestest1} \text{ on } \Lambda
\end{align} with
\begin{equation*}
a=\cfrac{k_2-k_1}{(\hat{R}-\check{R})^2}, \quad b=\cfrac{-2\check{R}(k_2-k_1)}{(\hat{R}-\check{R})^2}, \quad c=k_1+\cfrac{(k_2-k_1)\check{R}^2}{(\hat{R}-\check{R})^2}
\end{equation*}
This reduced problem corresponds to an equi-dimensional problem satisfying our modeling assumptions and having a constant solution equal to $k_1$ inside the cylinder. Further, the flux through the interface is zero, as the solution is $C^1$ in the whole domain.
Results are obtained considering a cube of edge $l=2$ ($\hat{R}=\sqrt{2}$) and choosing $\check{R}=0.01$, $k_1=0.5$ and $k_2=5$. Homogeneous Neumann boundary conditions are imposed on $\partial\Omega_{+}$ and $\partial \Omega_{-}$, whereas Dirichlet boundary conditions are imposed on $\partial\Omega_l$. Dirichlet boundary conditions equal to $k_1$ are imposed on segment endpoints.
Figure~\ref{section_view} on the right shows the approximated solutions $U$, $\hat{U}$ obtained inside the cube and on the segment for $h=0.086$ and $\hat{\delta}_u=1$, corresponding to $N=3715$ DOFs in the cube and $\hat{N}=57$ DOFs on the segment. The other parameters are $\delta_{\phi}=0.5$ and $\delta_{\psi}=0.5$. Convergence curves of the error can be computed and, given the regularity of the solution, optimal convergence trends are expected for the used finite element approximation. Let us introduce the errors $\mathcal{E}_{L^2}$, $\mathcal{E}_{H^1}$, for the 3D problem and $\widehat{\mathcal{E}}_{L^2}$ and $\widehat{\mathcal{E}}_{H^1}$ for the 1D problem, defined as follows:
 \begin{align*}
 &\mathcal{E}_{L^2}=\cfrac{||u_{ex}-U||_{L^2(\Omega)}}{||u_{ex}||_{L^2(\Omega)}} ,\qquad \mathcal{E}_{H^1}=\cfrac{||u_{ex}-U||_{H^1(\Omega)}}{||u_{ex}||_{H^1(\Omega)}},\\
  &\widehat{\mathcal{E}}_{L^2}=\cfrac{||\hat{u}_{ex}-\hat{U}||_{L^2(\Lambda)}}{||\hat{u}_{ex}||_{L^2(\Lambda)}}, \qquad \widehat{\mathcal{E}}_{H^1}=\cfrac{||\hat{u}_{ex}-\hat{U}||_{H^1(\Lambda)}}{||\hat{u}_{ex}||_{H^1(\Lambda)}}.
 \end{align*}
Figure~\ref{errtest1} displays the convergence trends for the above quantities against mesh refinement. Four meshes are considered, characterized by mesh parameters $h=0.215,0.136,0.086,0.054$ respectively, corresponding to $N=229,933,3715,14899$ DOFs and $\hat{N}=20,31,57,79$ DOFs ($\hat{\delta}_u=1$, $\delta_{\phi}=\delta_{\psi}=0.5$), confirming the expected behaviors. 

\begin{figure}
	\centering
	\includegraphics[width=0.45\linewidth]{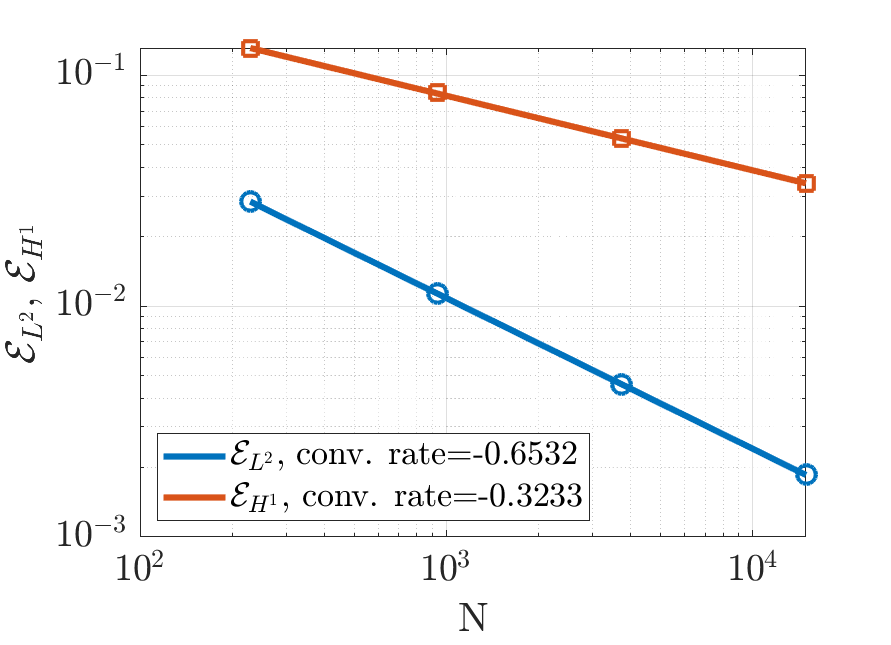}%
	\includegraphics[width=0.45\linewidth]{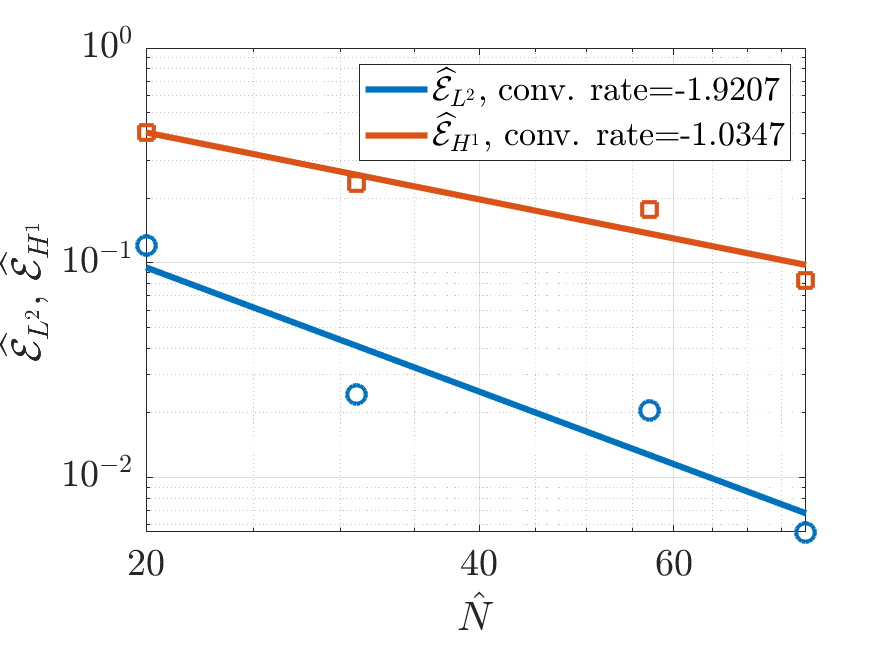}
	\caption{TP1: trend of the $L^2$ and $H^1$-norms of the relative errors under mesh refinement. On the left: error committed on the cube with respect to (\ref{uestest1}); on the right: error committed on the segment with respect to \eqref{uhatestest1}. Other parameters: $\hat{\delta}_{u}=1$, $\delta_{\psi}=\delta_{\phi}=0.5$.}
	\label{errtest1}
\end{figure}

For this simple problem it is possible to compute the condition number of the KKT optimality conditions \eqref{KKT}, and analyze how it is affected by different choices of the meshsize of the 1D meshes.  
Figure~\ref{cond1}, on the left, reports the condition number of the KKT system matrix as $\delta_\phi$ varies between $0.1$ and $1$, for five different values of $\hat{\delta}_u$ between $0.6$ and $1.4$, being instead $\delta_\psi=0.5$ fixed. We can see that $\delta_\phi$ has a relatively small impact on the conditioning of the system as long as $\delta_\phi<\hat{\delta}_u$ is chosen, otherwise a large rapid increase is observed as $\delta_\phi\geq\hat{\delta}_u$ grows.
It is therefore advisable to choose a quite coarse mesh for variable $\Phi$ with respect to the mesh for $\hat{U}$, even if no theoretical constraints emerged in the analysis. Figure~\ref{cond1}, on the right, shows again the conditioning of the KKT system matrix as $\delta_\psi$ varies between $0.1$ and $1$, for the same five values of $\hat{\delta}_u$ between $0.6$ and $1.4$, keeping this time $\delta_\phi=0.5$ fixed. It can be seen that, in this case, the conditioning is almost independent of $\delta_\psi$, for all the considered values of $\hat{\delta}_u$. It is however to remark that saddle point matrices as the one in \eqref{KKT} are typically ill conditioned. The use of a resolution strategy based on a gradient based scheme for the minimization of the unconstrained functional is expected to result in a problem with a mitigated condition number, actually coinciding with the application of a null-space based preconditioning technique \cite{Pestana2016}.

\begin{figure}
	\centering
	\includegraphics[width=0.44\linewidth]{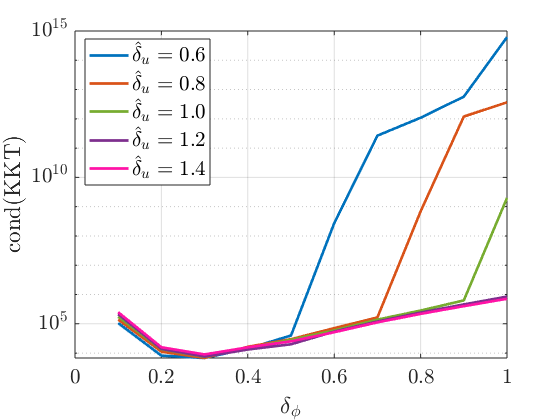}%
	\includegraphics[width=0.45\linewidth]{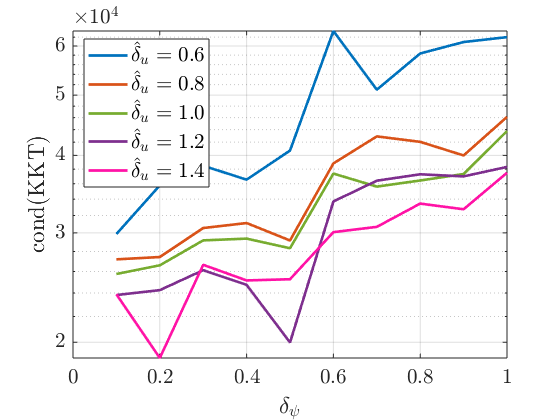}
	\caption{TP1: trend of the conditioning of the KKT system under the variation of the 1D mesh parameters. On the left variable $\delta_{\phi}$ and different values of $\hat{\delta}_u$, while $\delta_{\psi}=0.5$. On the right variable $\delta_{\psi}$ and $\delta_{\phi}=0.5$. In both cases $h=0.086$.}
	\label{cond1}
\end{figure}

\subsection{Test Problem 2 (TP2)}
The second example is set on a domain equal to the one of TP1. We consider three different problems defined as in \eqref{equaz1}-\eqref{condiz2}, characterized by three different values of the coefficient $\bm{\tilde{K}}$, equal to $1$, $10^2$ and $10^5$ respectively, whereas $\bm{K}=1$, $f=1$, $\overline{\overline{g}}=0$ are fixed for all the problems. Even the boundary conditions are shared: being $\partial\Omega_{+}$, $\partial\Omega_{-}$, and $\partial \Omega_{l}$ defined as previously, homogeneous Dirichlet boundary conditions are prescribed on $\partial\Omega_{+}$, $\partial\Omega_{-}$ and at segment endpoints, while homogeneous Neumann boundary conditions are set on $\partial \Omega_{l}$.

\begin{figure}
\centering
\includegraphics[width=0.99\textwidth]{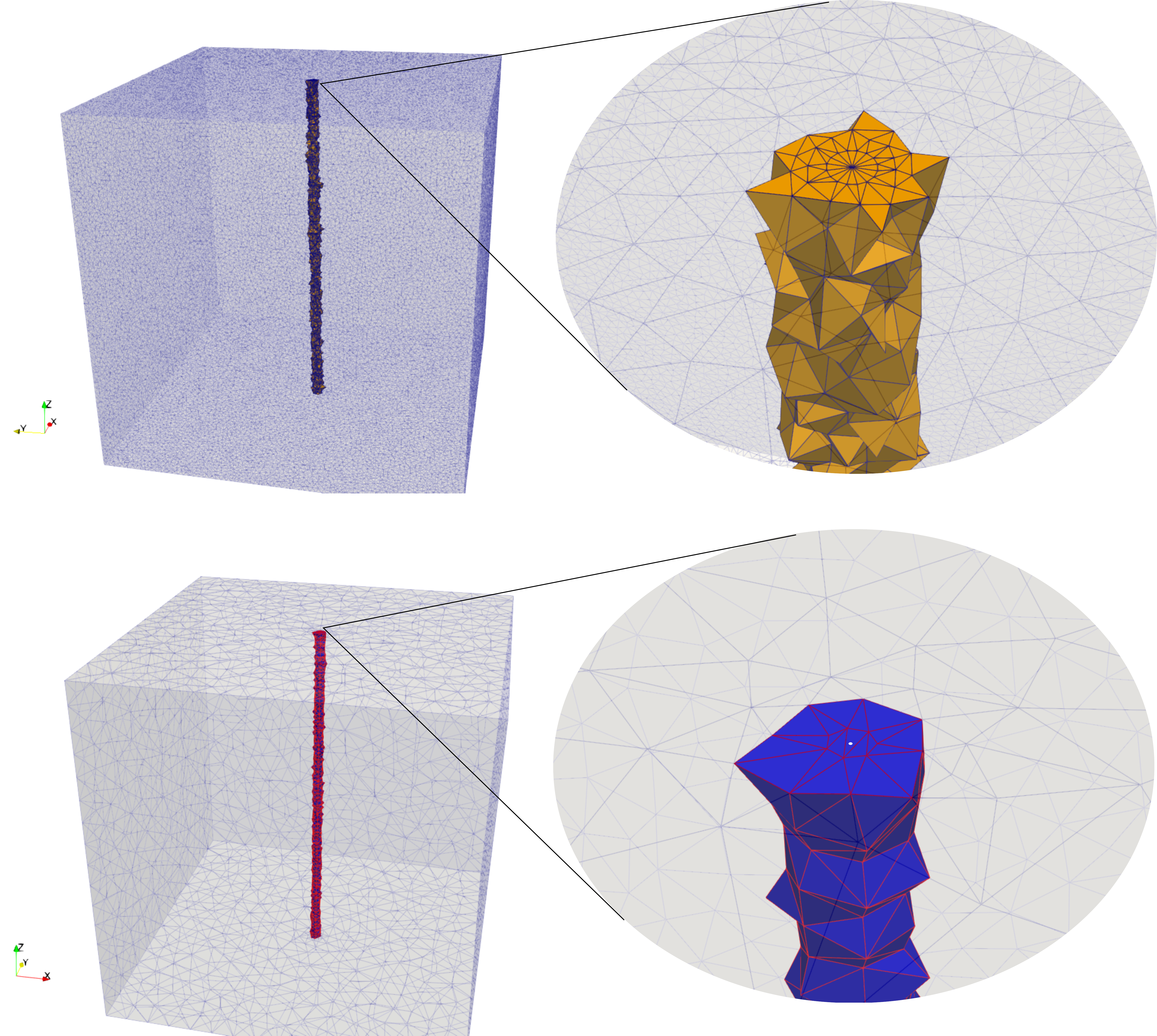}
\caption{TP2: Top: mesh used for the reference equi-dimensional problem, conforming to the cylindrical inclusion. Bottom: adapted non conforming mesh for the 3D-1D reduced problem with the proposed approach}
\label{3DMesh1}
\end{figure}

\begin{figure}
\centering
\includegraphics[width=0.45\textwidth]{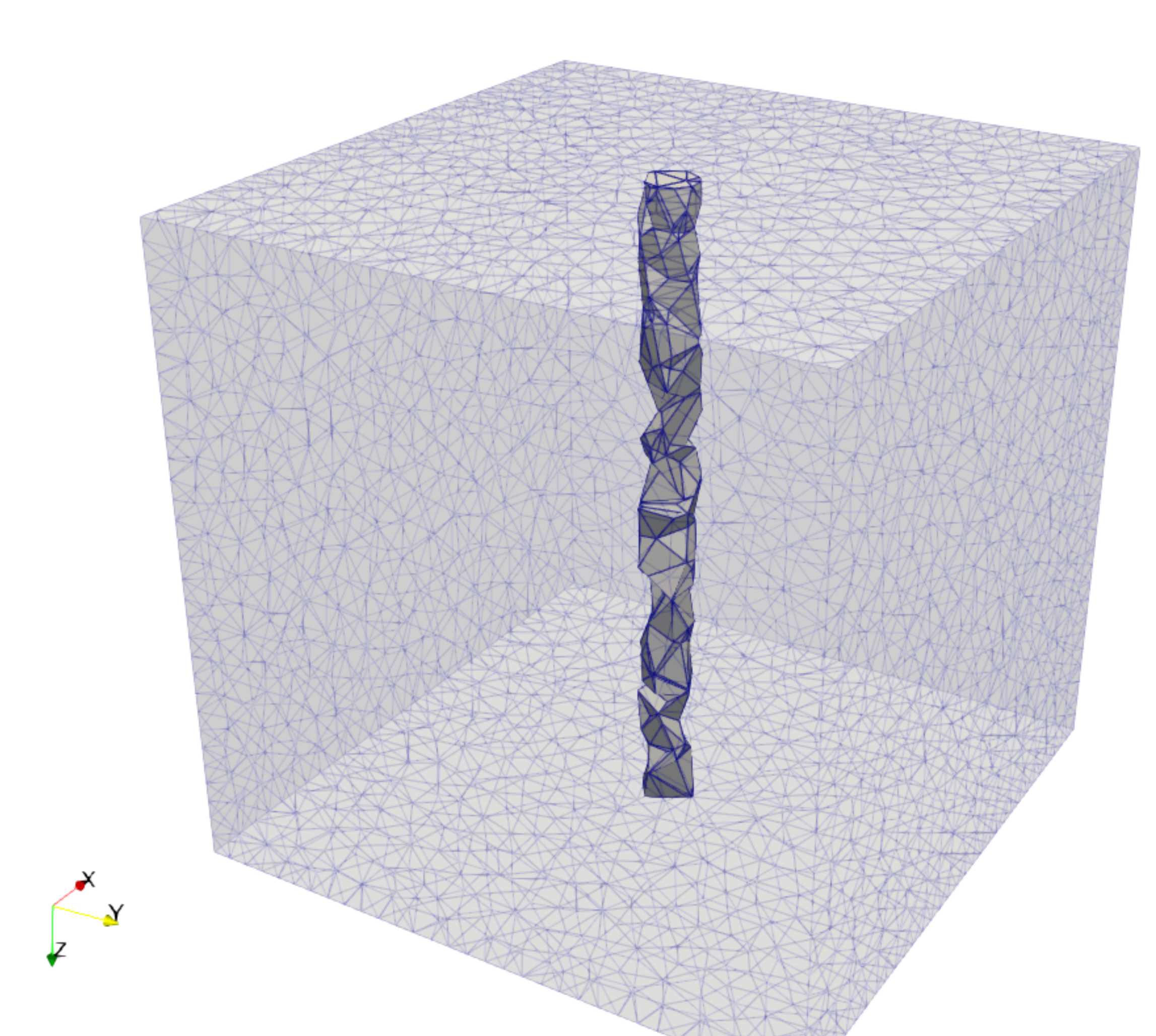}
\caption{TP2: Uniformly refined mesh with $h=0.086$.}
\label{3DMesh2}
\end{figure}

\begin{figure}
	\centering
	\includegraphics[width=0.45\linewidth]{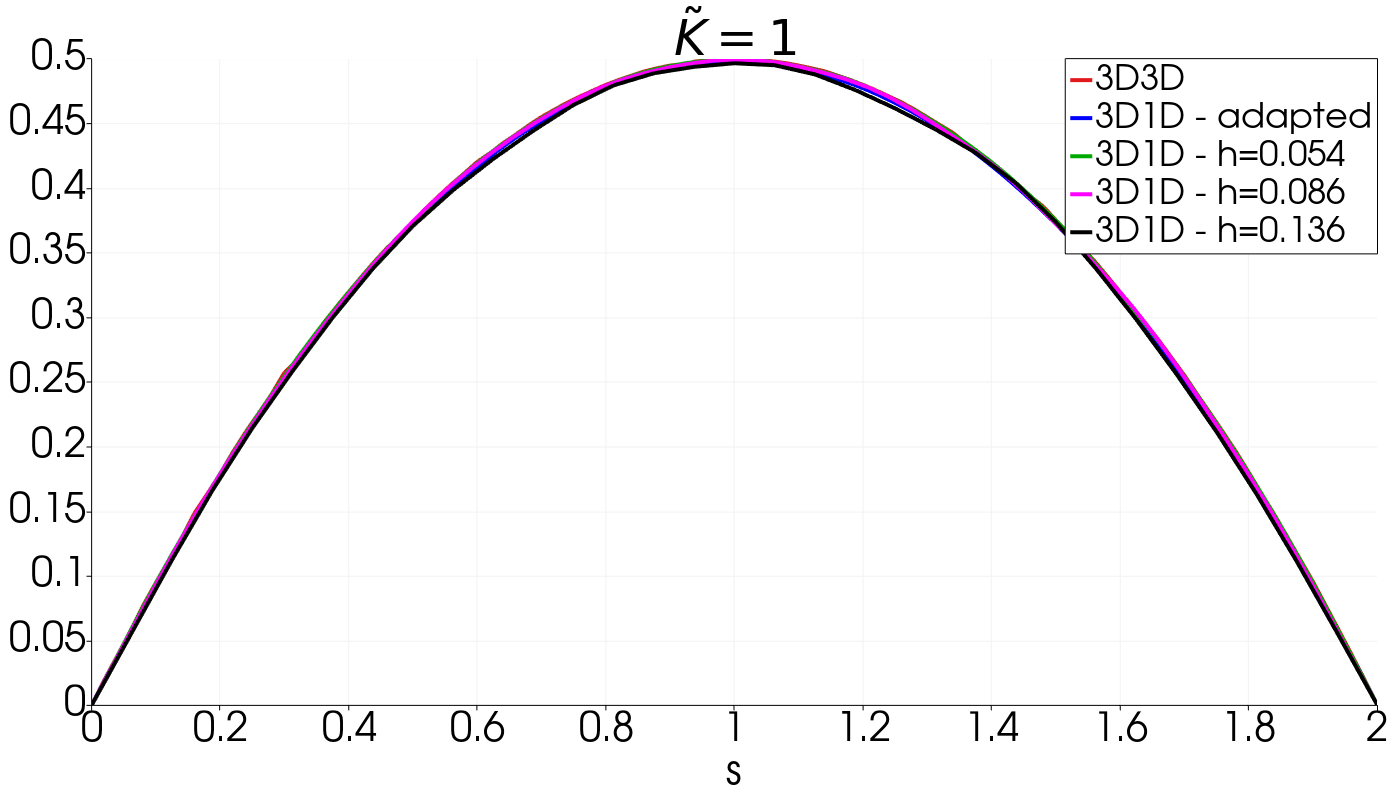}\hspace{5mm}%
	\includegraphics[width=0.45\linewidth]{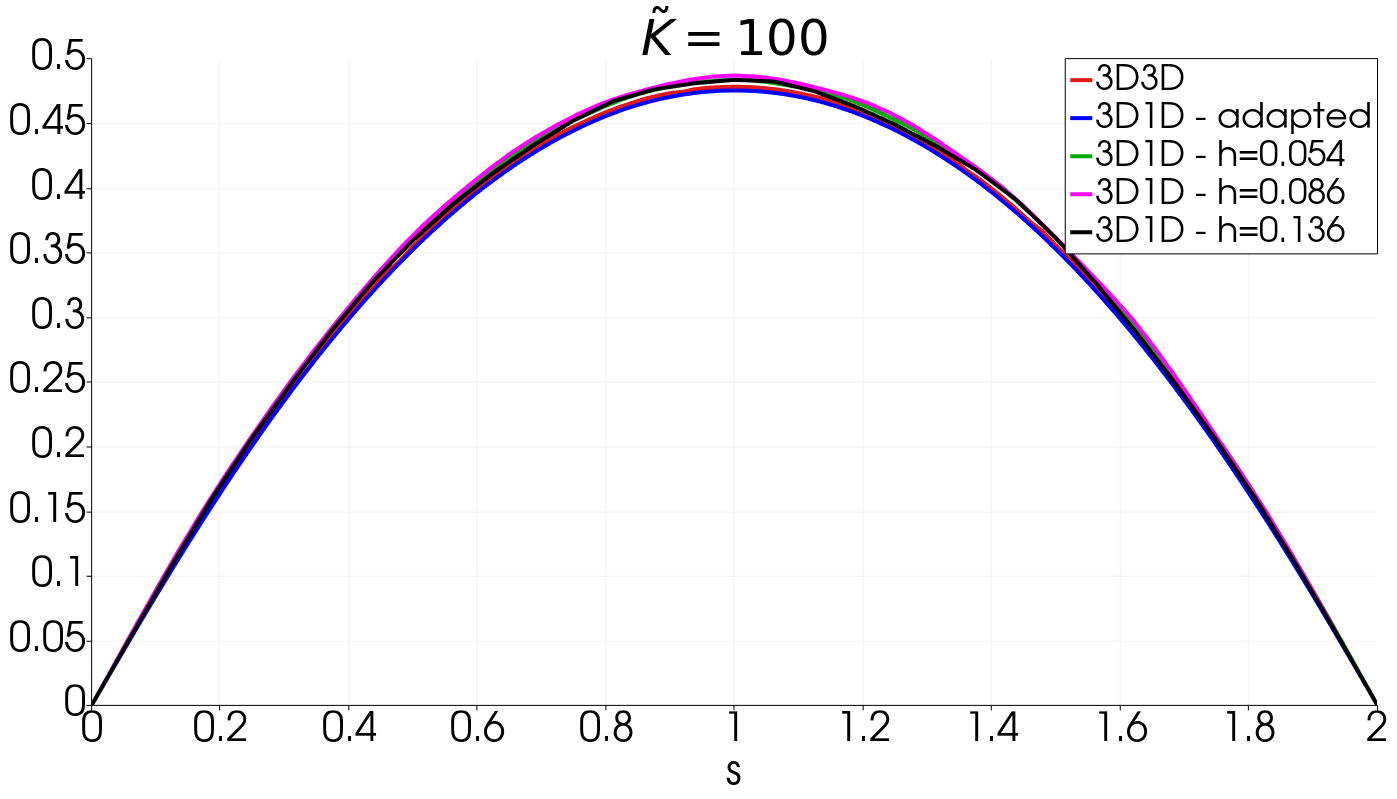}\medskip\\
	\includegraphics[width=0.45\linewidth]{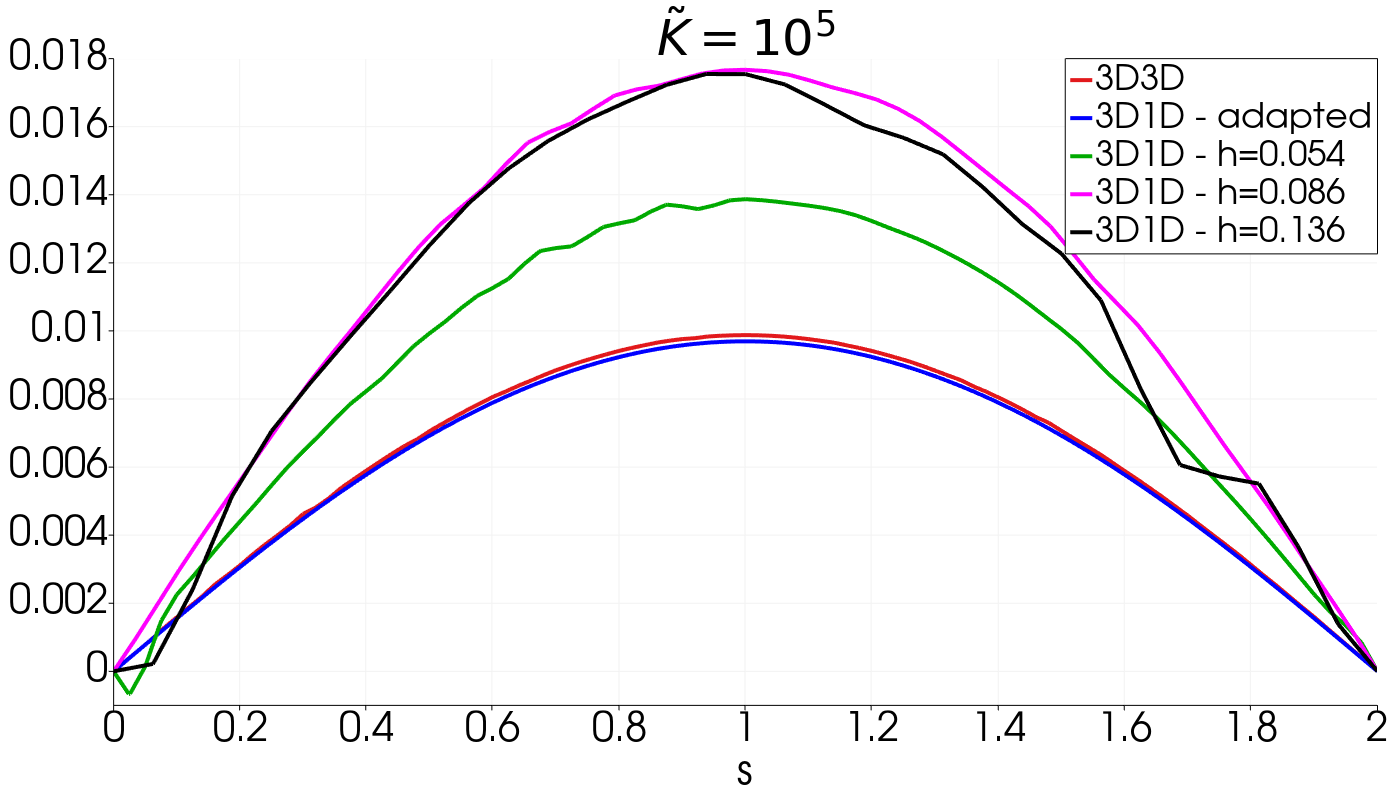}%
	\caption{TP2: comparison of the results obtained along the centerline of the cylinder in the 3D-3D conforming setting and by using the 3D-1D reduced model. Other parameters for the 3D-1D setting: $\hat{\delta}_u=1$ and $\delta_{\phi}=\delta_{\psi}=0.5$.}
	\label{confronto1D}
\end{figure}

\begin{figure}
\centering
\includegraphics[width=0.45\textwidth]{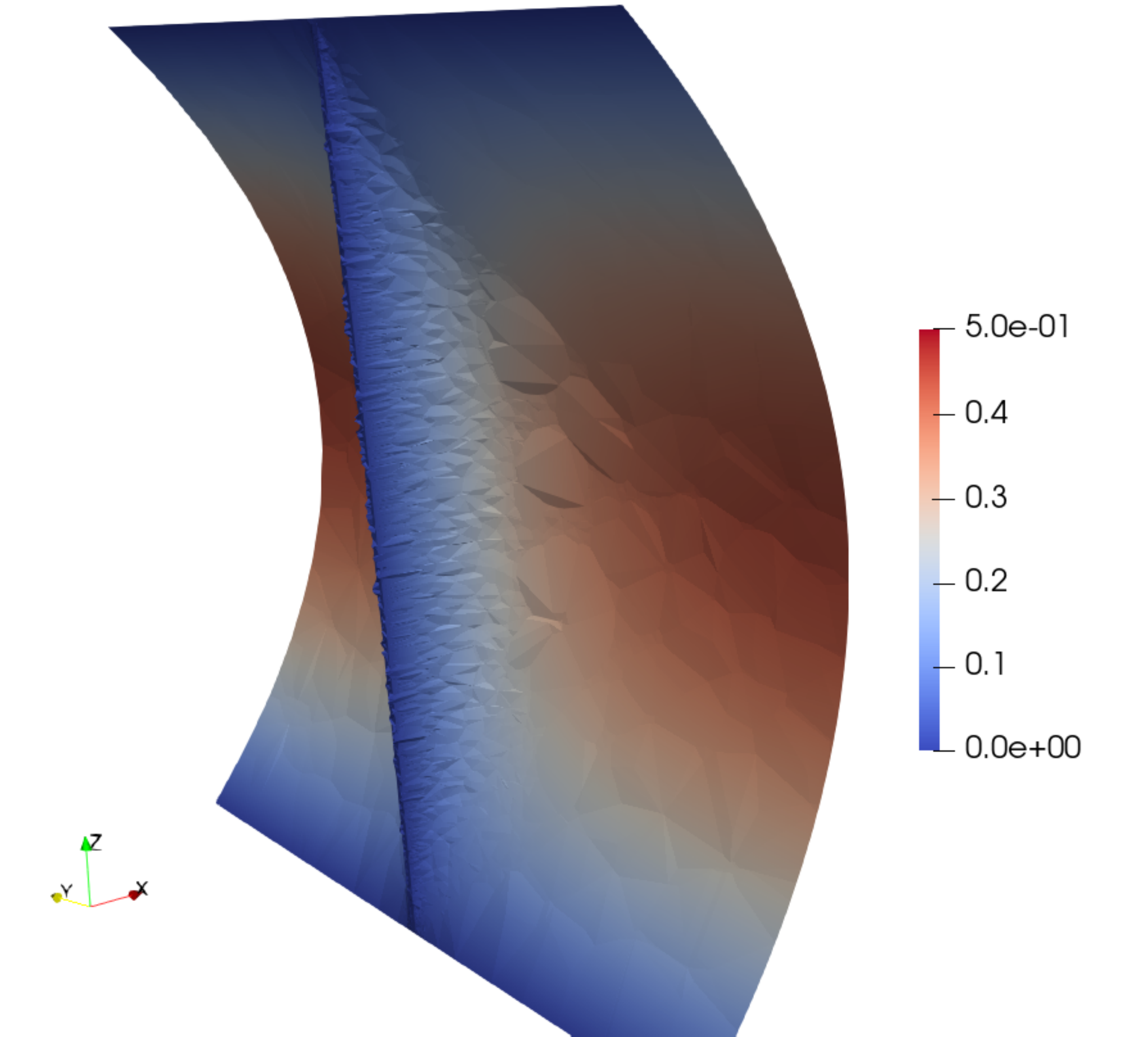}
\caption{TP2: comparison of the solution obtained on the adapted mesh for the 3D-1D problem with the reference solution on a plane parallel to $y-z$ and containing the centreline of the inclusion. $\bm{\tilde{K}}=10^5$.}
\label{SolComparison}
\end{figure}

The accuracy of the solution is evaluated by means of a comparison with an equi-dimensional problem having a cylindrical inclusion of radius $0.01$, with centreline coinciding with the 1D domain of the reduced problem. Dirichlet homogeneous boundary conditions are prescribed on the top and bottom faces of the 3D domains and homogeneous Neumann boundary conditions are set on the outer surface. A unitary forcing term is prescribed in the 3D domain outside of the cylindrical inclusion, where, instead a null forcing is set. As $\bm{\tilde{K}}$ grows, we move from a problem with a smooth solution to a problem with a jump of the gradient across the interfaces between the bulk 3D domain and the inclusion. 

The equi-dimensional problem is solved on a fine mesh, refined around the inclusion in order to match the geometry. The lateral surface of the cylindrical inclusion is represented as an extruded regular hexadecagon. The used mesh, along with a detail near to the inclusion, is shown in Figure~\ref{3DMesh1}, top. As this picture shows, the mesh is strongly refined close to the inclusion, in order to correctly catch its geometry and counts about $1.6$ million elements and $226288$ DOFs.

The reduced 3D-1D problem is solved on four different meshes: first a non conforming mesh, slightly refined close to the inclusion area, is considered, termed adapted mesh and having $2.8\times 10^4$ elements and $4890$ DOFs. This mesh is thus much coarser than the reference mesh. It is shown in Figure~\ref{3DMesh1}, at the bottom, along with a zooming of the zone around the 1D domain to highlight the non conformity. Further, three uniformly refined meshes are considered, with mesh parameters $h=0.136,0.086,0.054$, respectively, corresponding to $N=1287,4609,17164$ DOFs. The intermediate uniform mesh is shown in Figure~\ref{3DMesh2}.

The solution on the centreline of the inclusion obtained on the various considered meshes are compared to the reference solution on the centreline in Figure~\ref{confronto1D}, for $\bm{\tilde{K}}=1$ on the top left, $\bm{\tilde{K}}=100$ on the top right, and for $\bm{\tilde{K}}=10^5$ on the bottom. We can clearly see that, as long as the jump in the coefficient between the bulk domain and the inclusion is relatively small, the proposed approach correctly reproduces the solution on all the considered meshes. Instead, for large jumps, as it is for $\bm{\tilde{K}}=10^5$, the solution on the uniformly refined meshes are less accurate, whereas, the use of a slightly adapted mesh, even if still non conforming, is capable of producing a solution in very good agreement with the reference. The proposed approach can be thus of help in mitigating the overhead in mesh generation and to reduce problem size. A comparison between the reference solution and the solution of the reduced 3D-1D problem on the adapted mesh and with $\bm{\tilde{K}}=10^5$ is finally shown in Figure~\ref{SolComparison}, on a slice parallel to the $y-z$ plane and containing the centreline. The plot of the two solutions match well.

\subsection{Test with multiple intersecting inclusions (MI)}
As for the previous cases, let us consider a cube of edge $l=2$ centered in the axes origin. Let us then consider a set of $19$ inclusions of radius $\check{R}=10^{-2}$, whose centerlines intersect in $9$ points. We impose homogeneous Dirichlet boundary conditions on all the faces of the cube and at the dead ends of the network intersecting cube top and bottom faces, as shown in Figure~\ref{config19seg}. Homogeneous Neumann boundary conditions are imposed at  segment endpoints lying inside the cube. We consider a problem in the same form of \eqref{eq1loc}-\eqref{eq2loc}, with $i=1,...,\mathcal{I}=19$, spanning the segments, that form a unique connected component, as discussed in Section~\ref{Discrete}. Further, we consider $\bm{K}=1$, $f=0$ and $\bm{\tilde{K}}_i=100$, $\overline{\overline{g}}_i=3.14e-2$ $\forall i=1,...,19$.
\begin{figure}
	\centering
	\includegraphics[width=0.6\linewidth]{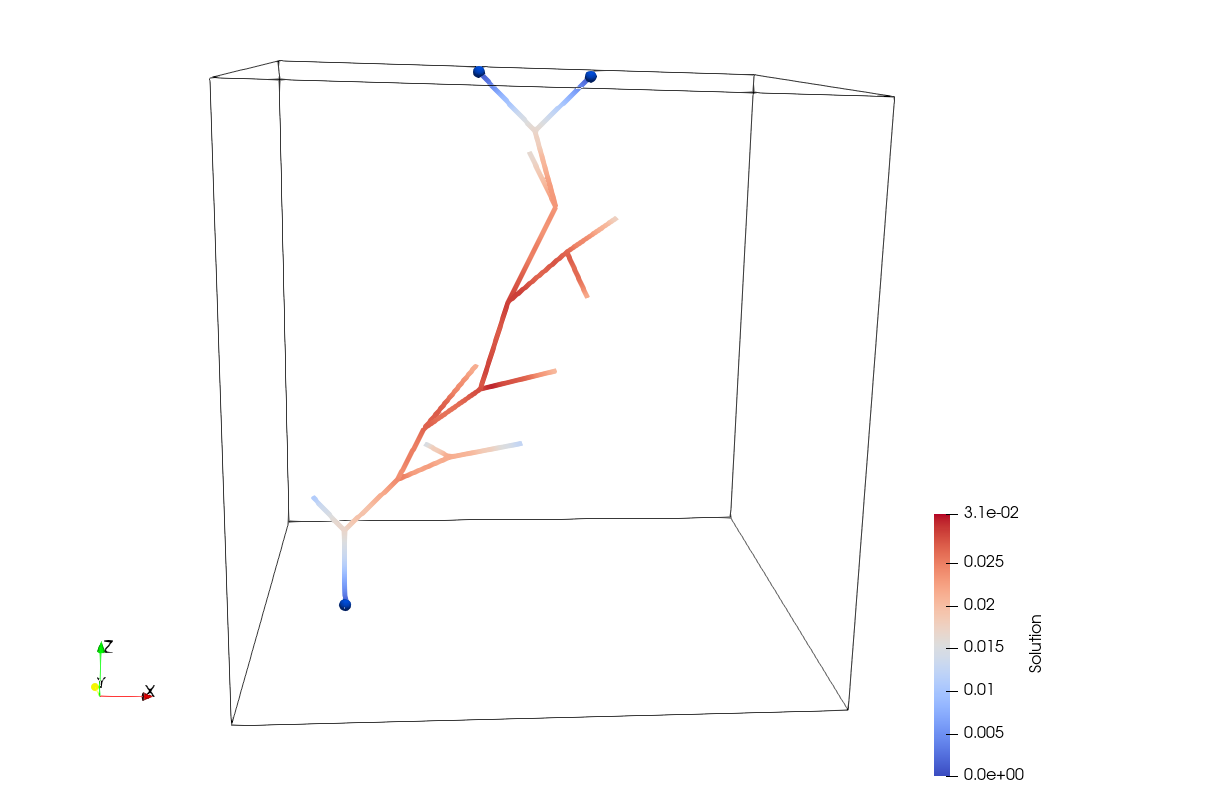}
	\caption{MI: Solution obtained on the centerlines of the inclusions for $h=0.054$, $\hat{\delta}_u=1$ and $\delta_{\phi}=\delta_{\psi}=0.5$. Homogeneous Dirichlet boundary conditions are imposed at the points marked in blue.}
	\label{config19seg}
\end{figure}

We refer again to Figure~\ref{config19seg} for the solution $\hat{U}$ obtained in the segment network for $h=0.056$ and $\hat{\delta}_{u}=1$, corresponding to $N=12873$ and $\hat{N}=309$ DOFs. The other parameters are $\delta_{\phi}=\delta_{\psi}=0.5$. Figure~\ref{sliceswarp}, on the left, shows instead, for the same choice of parameters, the solution $U$ obtained inside the cube, on three different slices, all parallel to the $x-y$ plane and located at $z=-0.5$, $z=0$ and $z=0.5$. On the right of Figure~\ref{sliceswarp} the behavior of the vector field $-\bm{K}\nabla U$ is shown.
\begin{figure}
	\centering
	\includegraphics[width=0.5\linewidth]{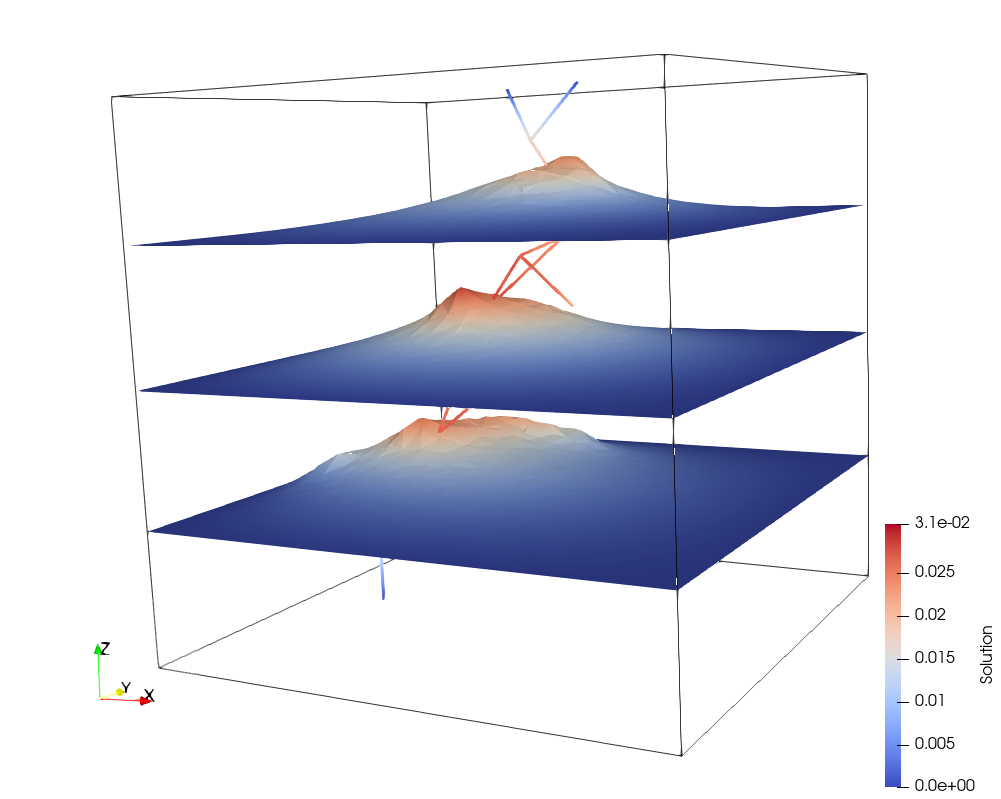}%
	\includegraphics[width=0.5\linewidth]{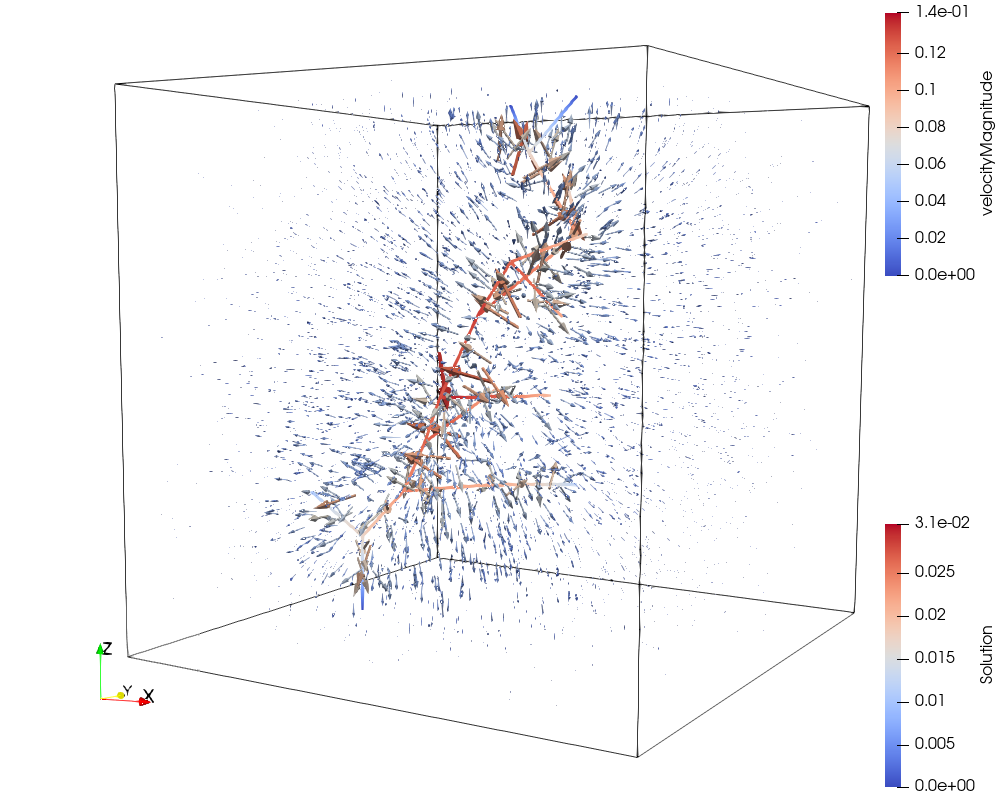}
	\caption{MI: On the left, solution obtained inside the cube on three different slices parallel to the $x-y$ plane and located at $z=-0.5$, $z=0$ and $z=0.5$. On the right focus on the behavior of the $-\bm{K}\nabla U$ vector field. The chosen mesh parameters are $h=0.054$, $\hat{\delta}_u=1$, $\delta_{\phi}=\delta_{\psi}=0.5$. }
	\label{sliceswarp}
\end{figure}

A quantitative evaluation of the numerical solution is provided through an error indicator, denoted by $\Delta_u^{L^2}$, measuring how well the continuity condition, enforced through the minimization of the functional \eqref{Jtilde}, is satisfied. We thus introduce the quantity:
\begin{equation}
\Delta_u^{L^2}=\cfrac{\sqrt{\sum_{i=1}^{\I}||U_{|_{\Lambda_i}}-\hat{U}_i||^2_{L^2(\Lambda_i)}}}{|\max(U,\hat{U})|\sqrt{l_{tot}}} \label{cont}
\end{equation}
resulting in a non dimensional number, with $l_{tot}$ denoting the total length of the segments in the domain. 

The trend of $\Delta_u^{L^2}$ against $\delta_{\phi}$ and $\delta_{\psi}$, both ranging between $0.1$ and $1$, is shown in Figure~\ref{contMI} on the left, still considering $h=0.054$ and ${\hat{\delta}_{u}}=1$. As expected, the error indicator decreases as the two parameters increase. The impact of $\delta_{\phi}$ appears to be stronger: for values of $\delta_{\psi}$ close to $1$, almost two orders of magnitude are swept by $\Delta_u^{L^2}$ as $\delta_{\phi}$ varies. The impact of $\delta_{\psi}$ on the continuity condition appears to be weaker, even if it can be seen that it becomes more relevant for high values of $\delta_{\phi}$, with almost one order of magnitude swept by the error indicator as $\delta_{\psi}$ increases. Figure~\ref{contMI}, on the right, shows instead the trend of $\Delta_u^{L^2}$ against $\hat{\delta}_u$, ranging between 0.6 and 2, for four values of the mesh size $h$, namely $h=0.215,0.136,0.086,0.054$, corresponding to $N=126,646,2951,12873$. The other parameters are $\delta_{\phi}=\delta_{\psi}=0.5$. We can see that the continuity error indicator is only marginally affected by the value of $\hat{\delta}_u$, whereas, it can be arbitrarily reduced by mesh refinement. It should be noted, however, that, for a fixed value of $\hat{\delta}_u$, a refinement of the 3D mesh also implies a refinement of the 1D mesh for $\hat{u}$, whereas, changes in $\hat{\delta}_u$, at fixed $h$ only refine the mesh of $\hat{u}$, leaving the 3D mesh unchanged.
\begin{figure}
	\centering
	\includegraphics[width=0.45\linewidth]{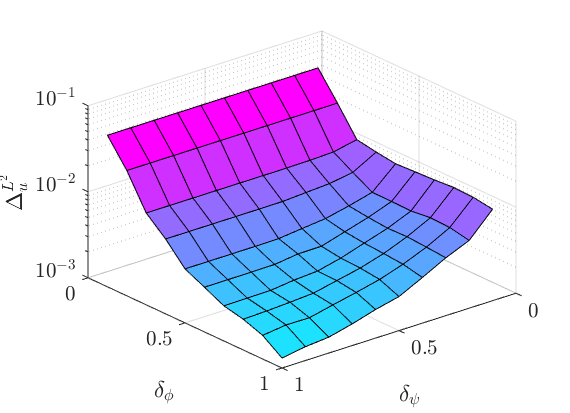}%
	\includegraphics[width=0.45\linewidth]{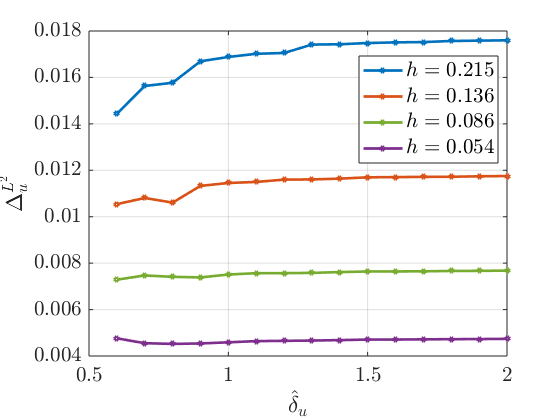}
	\caption{MI: Value of $\Delta_u^{L^2}$ \eqref{cont} under the variation of the mesh parameters. On the left variable $\delta_{\phi}$ and $\delta_{\psi}$, $h=0.054$; on the right, variable $\hat{\delta}_{u}$ and four different values of $h$, $\delta_{\phi}=\delta_{\psi}=0.5$.}
	\label{contMI}
\end{figure}

\section{Conclusions}
A novel approach for 3D-1D coupled problems has been proposed. The method derives from a mathematical formulation in proper functional spaces that allows the definition of a well posed trace operator from functions in the three dimensional space to one dimensional manifolds. The 1D problems are decoupled from the problem on the bulk 3D domain and two interface variables are introduced in this domain decomposition process, thus resulting in a three field formulation of the original problem. A cost functional is introduced and minimized to impose matching conditions at the interfaces. The method allows to enforce continuity conditions and flux balance at the interfaces between sub-problems on non-conforming meshes, thus strongly alleviating the mesh generation process. Indeed meshes on the various sub-domains can be independently generated. Numerical results on two simple test problem and on a more complex configuration show the viability of the proposed approach and are used to analyze the effect of method parameters on the condition number of the discrete problem and on solution accuracy. 

A formulation suitable for efficient resolution through iterative gradient-based schemes is also envisaged and should be further investigated to allow simulation on problems of high geometrical complexity.

\section*{References}
\bibliographystyle{elsarticle-num}
\bibliography{3D1DBib}

\end{document}